\documentclass[USenglish,cleveref, cleveref, thm-restate, oneside]{amsart}

\usepackage[pagewise]{lineno}

\usepackage[a4paper, total={6.5in, 9in}]{geometry}
\usepackage{url}
\usepackage{hyperref}
\usepackage{amsthm, amssymb, amsmath}
\usepackage{mathtools}
\usepackage{thm-restate}

\usepackage{caption}
\usepackage{multirow}
\usepackage{float}
\usepackage{algpseudocodex}

\usepackage{enumitem}

\newcommand{\Triv}{1}
\newcommand{\canset}{X\xspace}
\newcommand{\fillin}{\textsc{FillIn}\xspace}
\newcommand{\testgeo}{\textsc{TestGeodetic}\xspace}
\newcommand{\handlethreecolls}{\textsc{HandleCollisions}\xspace}

\renewcommand{\setminus}{\mysetminus}
\newcommand{\mysetminusD}{\raisebox{.3pt}{\hbox{\tikz{\draw[line width=0.6pt,line cap=round] (3.5pt,0pt) -- (0,5.2pt);}}}}
\newcommand{\mysetminusT}{\mysetminusD}
\newcommand{\mysetminusS}{\raisebox{.5pt}{\hbox{\tikz{\draw[line width=0.45pt,line cap=round] (2.2pt,0) -- (0,3.8pt);}}}}
\newcommand{\mysetminusSS}{\raisebox{.35pt}{\hbox{\tikz{\draw[line width=0.4pt,line cap=round] (1.5pt,0) -- (0,2.8pt);}}}}

\newcommand{\mysetminus}{\mathbin{\mathchoice{\mysetminusD}{\mysetminusT}{\mysetminusS}{\mysetminusSS}}}

\usepackage[textwidth=1.8cm]{todonotes}
\usepackage{xspace}
\usepackage[capitalize]{cleveref}
\usepackage{amsfonts}
\usepackage{booktabs}

\newcommand{\abs}[1]{\left|\mathinner{#1}\right|}

\newcommand{\gen}[1]{\left< \mathinner{#1} \right>}

\newcommand{\ord}{\operatorname{ord}}
\newcommand{\ov}[1]{\overline{#1}}

\newcommand\ie{i.e\@., }

\newcommand{\sse}{\subseteq}
\newcommand{\cN}{\mathcal{N}}
\newcommand{\cB}{\mathcal{B}}

\DeclareMathOperator{\Cay}{Cay}

\newtheorem{theorem}{Theorem}[section]
\newtheorem{proposition}[theorem]{Proposition}
\newtheorem{corollary}[theorem]{Corollary}
\newtheorem{lemma}[theorem]{Lemma}

\newtheorem{observation}[theorem]{Observation}

\newtheorem{theoremx}{Theorem}

\newtheorem{conjecturex}[theoremx]{Conjecture}  

\newtheorem{definition}[theorem]{Definition}

\theoremstyle{remark}
\newtheorem{remark}[theorem]{Remark}

\newfloat{algorithm}{t}{lop}
\floatname{algorithm}{Algorithm}

\begin{document}

\title{Finite groups with geodetic Cayley graphs}

\author[M. Elder]{Murray Elder}
\address{School of Mathematical and Physical Sciences, University of Technology Sydney, Broadway NSW 2007, Australia}
\email{murray.elder@uts.edu.au}

\author[A. Piggott]{Adam Piggott}
\address{Mathematical Sciences Institute, Australian National University, Canberra ACT  2601, Australia}
\email{adam.piggott@anu.edu.au}

\author[F. Stober]{Florian Stober}
\address{Institut für Formale Methoden der Informatik, Universität Stuttgart, Universitätsstraße 38, 70569 Stuttgart, Germany}
\email{florian.stober@fmi.uni-stuttgart.de}

\author[A. Thumm]{Alexander Thumm}
\address{Department für Elektrotechnik und Informatik,
Universität Siegen,
Hölderlinstrasse 3,
D-57076 Siegen, Germany}
\email{alexander.thumm@uni-siegen.de}

\author[A. Wei\ss]{Armin Wei\ss}
\address{Institut für Formale Methoden der Informatik, Universität Stuttgart, Universitätsstraße 38, 70569 Stuttgart, Germany}
\email{armin.weiss@fmi.uni-stuttgart.de}

\begin{abstract}
A connected undirected graph is called \emph{geodetic} if for every pair of vertices there is a unique shortest path connecting them. 
It has been conjectured that for finite groups, the only geodetic Cayley graphs are odd cycles and complete graphs. In this article we present a series of theoretical results which contribute to a computer search verifying this conjecture for all groups of size up to 1024. The conjecture is also verified  for several infinite families of groups including dihedral and some families of nilpotent groups. Two key results which enable the computer search to reach as far as it does are: if the center of a group has even order, then the conjecture holds (this eliminates all $2$-groups from our computer search); if a Cayley graph is geodetic then there are bounds relating the  size of the group, generating set and  center (which significantly cuts down  the number of  generating sets which must be searched).
\end{abstract}

\date{\today}

\makeatletter
\@namedef{subjclassname@2020}{\textup{2020} Mathematics Subject Classification}
\makeatother



\maketitle

\section{Introduction}
Cayley graphs form an important subclass of vertex-transitive and regular graphs. 
The undirected Cayley graph of a group $G$ with respect to a generating set $\Sigma$ is the connected graph on vertex set $V = G$ and edge set $E$, where $\{g, h\} \in E$ if and only if there is a generator $a \in \Sigma$ such that $ga = h$.
Besides being an important tool in combinatorial group theory, there are also interesting graph-theoretic questions about Cayley graphs.
One example which has been much studied is the longstanding conjecture that every finite undirected Cayley graph that is not the complete graph on two vertices has a Hamiltonian cycle (see for example \cite{PAK20095501}). 

A connected undirected graph $\Gamma = (V, E)$ is called \emph{geodetic} if for any pair $u, v \in V$, the shortest path from $u$ to $v$ is unique.
Research on geodetic graphs began in 1962, when Ore posed the problem of classifying all such graphs~\cite{ore1962theory}.
This goal has been achieved so far for planar geodetic graphs, and geodetic graphs of diameter two~\cite{stemple1968planar,stemple1974geodetic}; yet, after decades of active research, a full classification of finite geodetic graphs has not been attained  (for some recent developments, see for example \cite{etgar2022connectivity,frasser2020geodetic,StoberW2023arXiv}).

In 1997, Shapiro~\cite{shapiro1997pascal} asked whether each 
finitely generated group that admits a geodetic Cayley graph with respect to some finite generating set is  
\emph{plain}, that is, isomorphic to the free product of finitely many finite groups and finitely many copies of $\mathbb{Z}$. 
It is well known and not hard to see that the converse holds: each plain group admits a geodetic Cayley graph (with respect to the generating set consisting of each non-trivial element of each finite factor, and a cyclic generator for each $\mathbb{Z}$ factor).
Recently, significant progress has been made on  this question 
and variants of it~\cite{EisenbergP19,elder2022rewriting,elder2023kgeodetic,Spriano} (see also the paragraph below on related work). 

For any group, the Cayley graph with respect to the generating set consisting of every non-trivial element is the complete graph, which is geodetic.
For cyclic groups of odd order, there is a second possibility:
taking an arbitrary single generator, the Cayley graph is an odd cycle, which is also geodetic.
The question we study here is whether there is any other possibility.
In a 2017 PhD thesis \cite{wrap108882}, Federici conjectured that among the finite groups, there is none.

\begin{conjecturex}[{\cite[Conjecture 6]{wrap108882}}]
\label{conj:main}
	Every finite geodetic Cayley graph is either a cycle of odd length or a complete graph.
\end{conjecturex}
We say that a group satisfies \cref{conj:main} if all its possible Cayley graphs satisfy \cref{conj:main}.
In this paper we report on a systematic computer experiment which confirms the conjecture for a significant number of groups.

\begin{restatable}{theoremx}{ThmB}
\label{thm:mainCompSearch}
All groups of size up to 1024 satisfy \Cref{conj:main}.
\end{restatable}

We also show that all groups of even order up to 2014, all simple groups of order up to 5000, and the symmetric group $S_7$ satisfy \Cref{conj:main}.
Given that there are approximately $49.5$ billion groups of order at most $1024$, and each group of order $n$ has $2^{n-1}$ potential generating sets, a naive computer search could not possibly achieve the result in \Cref{thm:mainCompSearch}.
Instead, our computer search relies on a series of theoretical results concerning finite groups and when they can admit geodetic Cayley graphs. Some of these results confirm the conjecture for infinite families of groups, as summarized by the following.

\begin{theoremx}\label{thm:mainTheory}
    A finite group $G$ satisfies \Cref{conj:main} if any one of the following conditions holds.  
\begin{enumerate}
    \item The center $Z(G)$ of $G$ has even order (\Cref{cor:evencenter}).
    \item The group $G$ contains an abelian subgroup of index two (\Cref{thm:abelian_by_c2}).
    \item The group $G$ is nilpotent and its order does not have certain small divisors depending on the nilpotency class of $G$ (see \Cref{thm:nilpotent} for the precise statement).
    \item The group $G$ has large commutativity degree (\Cref{thm:commutativity_degree}).
\end{enumerate}
\end{theoremx}

This cuts down the number of groups which need to be considered enormously, from $49.5$ billion to $3197$. 
In fact, just excluding abelian groups and groups with even-order center leaves $4734$ groups of order at most $1024$.
We note that \Cref{conj:main} was shown to be satisfied by all abelian groups by Georgakopoulos as reported in \cite[Proposition 10]{wrap108882}.

A second key theoretical result is the following, which provides an upper bound on the number of generating sets that the search must consider for each group.

\begin{theoremx}\label{thm:bound_generating_set+simplified}
  Let $G$ be a finite group with generating set $\Sigma$ such that the Cayley graph $\mathrm{Cay}(G, \Sigma)$ is a counterexample to \Cref{conj:main}, \ie it is geodetic but neither complete nor a cycle.
  Then $\abs{\Sigma} < 1.07\sqrt{|G|}$.
\end{theoremx}

\noindent
\textit{Article organization.} Section~\ref{sec:prelim} sets our notation for groups and graphs, and provides some  preliminary results about geodetic graphs.
Theorems~\ref{thm:mainTheory} and~\ref{thm:bound_generating_set+simplified} 
are proved via a series of results presented in Sections~\ref{sec:theoryProofs} to~\ref{sec:furtherCases}. 
In Section~\ref{sec:experiments} we give details of the implementation of our computer program.
We use GAP \cite{GAP4} and its small group (SmallGrp) library \cite{SmallGrp} to enumerate the groups and check which of them are already covered by Theorem~\ref{thm:mainTheory}.
Then we apply an exhaustive search to check whether there is any geodetic generating set for any of the remaining groups. 
This exhaustive search crucially relies on several pruning methods based on variants of \cref{thm:bound_generating_set+simplified} and other theoretical results developed in Sections~\ref{sec:theoryProofs} and \ref{sec:bounds}. The code is available at

\centerline{\url{https://osf.io/9ay6s/?view_only=37e18301e4e74e12bfe4e07b90b924c0}.}

\medskip
\noindent
\textit{Related work.}
As noted above, the first mention of Conjecture~\ref{conj:main} that we are aware of is in the 2017 thesis of Federici \cite{wrap108882}, where the conjecture is proved for abelian groups and where it is shown that the Cayley graph of a semidirect product $C_m \rtimes C_n$ with $C_n$ acting faithfully on $C_m$ with respect to the ``standard'' generating set is not geodetic.
Notice that for both $m$ and $n$ sufficiently large, this class of groups turns out to be one of the most difficult cases for our computer search (see \cref{sec:exp_results}).
Besides this result, Federici proves some helpful lemmas for general geodetic Cayley graphs and runs computer experiments (unfortunately undocumented) which do not find any geodetic Cayley graphs except the obvious ones.

In 2022 Che \cite{Che} programmed an exhaustive computer search as part of an undergraduate project supervised by 
Alexey Talambutsa 
establishing \cref{conj:main} for subgroups of the symmetric group~$S_4$.
Moreover, for 
\begin{itemize}
    \item all subgroups of $S_{6}$ with generating sets of size five,
    \item all subgroups of $S_{7}$ with generating sets of size four,
    \item all subgroups of $S_{9}$ with generating sets of size three,
    \item all subgroups of $S_{10}$ with generating sets of size two,
\end{itemize}
Che showed that none of the corresponding Cayley graphs except complete ones and odd cycles are geodetic. For the source code, see \url{https://gitlab.com/andr0901/kayley-geodesics}.
Be aware that restricting the number of generators to five or less is a strong restriction. (Note that by contrast, \Cref{thm:mainCompSearch} verifies all generating sets for $S_6$,  and \Cref{thm:AdditonalLargeGroups} below verifies all generating sets for $S_7$. The groups $S_9$ and $S_{10}$ are beyond our scope.)

A geodetic Cayley graph for a group 
corresponds to an inverse-closed confluent length-reducing rewriting system \cite{elder2022rewriting}, yielding a connection between a geometric property (being geodetic) and  formal languages.  Monadic rewriting systems are special cases of length-reducing rewriting systems, and a result that is similar to \Cref{conj:main} is known for monadic rewriting systems: the only normalized finite confluent monadic rewriting systems for finite groups are those that correspond to complete Cayley graphs, or those that correspond to directed cycles for cyclic groups \cite[Corollary 3.13]{AdamMonadic}.  As discussed in \cite{elder2022rewriting}, it is an interesting open problem to classify the groups presented by inverse-closed finite convergent length-reducing rewriting systems.  One reason to pursue new examples of finite geodetic Cayley graphs is that, by a simple construction that corresponds to the free product of groups, they immediately give new examples of infinite geodetic Cayley graphs.

In \cite{elder2023kgeodetic} Townsend and the first two authors generalize the concept of a geodetic graph to graphs which have  at most $k$ different geodesics between any pair of vertices for some constant $k$,  which they call \emph{$k$-geodetic},
and study properties of groups which admit $k$-geodetic Cayley graphs.
While the main focus in \cite{elder2023kgeodetic} is on infinite groups, \cite[Example 1.1]{elder2023kgeodetic} gives examples of Cayley graphs of finite groups that are $k$-geodetic but not $(k-1)$-geodetic or complete for $k\geq 2$. The easiest such examples are cyclic groups of even order with a non-complete 2-geodetic Cayley graph as well as a complete bipartite Cayley graph.
 
A graph is called \emph{strongly geodetic} if for every pair of vertices there is at most \emph{one} non-backtracking path connecting them that has length at most the diameter of the graph \cite{BosakKZ68}. 
Clearly, all strongly geodetic graphs are geodetic. 
By \cite[Theorem 1]{BosakKZ68} the class of strongly geodetic regular graphs coincides with the class of so-called \emph{Moore graphs}, which also have been thoroughly studied in graph theory (for a definition see \cite{hoffman1960moore} or the survey paper \cite{MillerS05}). 
Moreover, by \cite{BannaiI73,Damerell73} the Moore graphs are completely classified: cycles of odd length, complete graphs, the Petersen graph, Hoffman-Singleton graph, and potentially hypothesized graphs (not known to exist) with 3250 vertices and degree 57. Of these, only the cycles and the complete graphs are Cayley graphs \cite{Cameron99,MillerS05} (we provide an alternative proof of this fact in \Cref{cor:notCG} below).
Hence, \Cref{conj:main} is true if we replace ``geodetic'' with ``strongly geodetic''.

\section{Preliminaries}\label{sec:prelim}

\subsection{Groups and words}
Throughout this article, we only consider finite groups.
We write these multiplicatively, mostly omitting the  
binary operation
altogether, and denote their identity element by $1$. 
As usual, we write $[g,h] = g^{-1}h^{-1}gh$ for the \emph{commutator} of two elements $g,h$ of a group $G$.
The \emph{center} of $G$ is the subgroup $Z(G) = \{g \in G \mid  [g,h] = 1 \text{ for all } h\in G\}$. 
Two elements $g, g' \in G$ are \emph{conjugate} if $g' = h^{-1}gh$ $ (\eqqcolon g^h)$ for some $h \in G$.
Given $g\in G$ and $H\subseteq G$, we write $g^H=\{g^h \mid h\in H\}$. 

We denote by $\operatorname{ord}(g)$ the \emph{order} of $g \in G$, \ie the smallest positive integer $n$ with $g^{n} = 1$.
The \emph{order} (number of elements) of the group $G$ is denoted by $\abs{G}$, and the \emph{exponent} of $G$ by $\exp(G)$, \ie the smallest positive integer $n$ such that $g^n=1$ for all $g\in G$. We denote the trivial group by $\Triv$.

\medskip

For a subset $X\sse G$, we write $\gen{X}$ for the subgroup generated by $X$.
It consists of those group elements that can be written as words over the alphabet $\Sigma = X \cup X^{-1}$.
We denote the set of all such words by $\Sigma^\ast$.
Given words $v,w \in \Sigma^\ast$, we write $v=w$ with the meaning that $v$ and $w$ evaluate to the same group element in $G$, whereas $v \equiv w$ denotes equality of words.

A subset $\Sigma \sse G$ with $\gen\Sigma = G$ is called a \emph{generating set} of $G$.
Throughout, we assume that all generating sets satisfy $1 \not \in \Sigma$ and are \emph{symmetric}, \ie  $a \in \Sigma$ implies $a^{-1} \in \Sigma$.
We sometimes represent the inverse $a^{-1}$ of a generator $a \in \Sigma$ by $\ov a$ to emphasize that it is a single letter.

The length of a word $w = a_1 \cdots a_n \in \Sigma^*$ (with $a_i \in \Sigma$) is denoted by $\abs{w} = n$. 
We denote the set of words over $\Sigma$ of length $n$ by $\Sigma^n$.
A word $w \in \Sigma^*$ is called a \emph{geodesic} for (or representing) a group element $g \in G$ if $w= g$ and $w$ is shortest among all words with that property.
The geodesic length of $g \in G$ is defined as the length of a geodesic word representing $g$.
If $g$ admits a unique geodesic, we denote its geodesic by $\mathrm{geod}(g)$.

\medskip
 
We use the following notation for specific groups. $C_n$ is the cyclic group of order $n$, $D_{2n}$ the dihedral group of order $2n$,
and $S_n$ (resp.\ $A_n$) the symmetric (resp.\ alternating) group on $n$ elements.
A direct product is denoted by $\times$ and a \emph{semidirect product} by $\rtimes$, where 
$G = N \rtimes H$ means that $N$ is a normal subgroup of $G$ and $H$ a subgroup of $G$ such that $G = NH$ and $N\cap H = \Triv$. Notice that $G/N \cong H$ in this case. 
Be aware that, if we are given only $N$ and $H$, then writing $G = N \rtimes H$ does not completely define $G$~-- instead we need to also specify a homomorphism from $H$ to the automorphism group of $N$ describing the action $(h,n) \mapsto hnh^{-1}$ of $H$ on $N$ (in other words $H$ acts on $N$ via automorphisms).

\subsection{Graphs and Cayley graphs}

We consider only \emph{undirected finite simple graphs} $\Gamma = (V,E)$ where $E \sse \binom{V}{2}$.
The \emph{Cayley graph} $\Cay(G,\Sigma) = (V,E)$ of a group $G$ (with respect to a generating set $\Sigma \sse G$) is defined by $V = G$ and $E = \{\{g,ga\} \mid g \in G, a\in \Sigma\}$. 
In other literature the \emph{directed} Cayley graph is often considered; however, throughout this paper $\Cay(G,\Sigma)$ is undirected (and contains no loops) due to the above assumptions on $\Sigma$.
Note that $G$ acts on $\Cay(G,\Sigma)$ by left multiplication, \ie $g . v = gv$ and $g . \{h,ha\} = \{gh,gha\}$.
In particular, $\Cay(G,\Sigma)$ is vertex-transitive.
The Cayley graph is connected because $\Sigma$ generates $G$.
In fact, it is biconnected for every finite group $G$ with $G\neq \Triv$ and $G \not\cong C_2$ \cite[Theorem~3]{Watkins70}. The degree of each vertex is $\abs{\Sigma}$.

\medskip

Given a graph $\Gamma = (V, E)$, a \emph{path} of length $n$ from $v_0 \in V$ to $v_n \in V$ in $\Gamma$ is a sequence $v_0,\ldots,v_n$ of vertices (not necessarily distinct) with $\{v_{i-1},v_i\} \in E$ for all $1 \leq i \leq n$.
A \emph{cycle} is a path of length at least 3 with $v_i=v_j$ if and only if $\{i,j\}=\{0,n\}$.
The \emph{distance} $d(u,v)$ between vertices $u$ and $v$ is defined as the length of a shortest path (\emph{geodesic}) connecting $u$ and $v$.
The \emph{diameter} of $\Gamma$ is $\max \{ d(u, v) \mid u,v \in V \}$.
Moreover, given $v \in V$ and $H\sse V$, we define $d(v, H) = \min \{ d(v, h) \mid h \in H \}$ to be the distance of $v$ to $H$, $\cN(v) = \{u \mid \{u, v\} \in E\}$ to be the neighborhood of $v$, and $\cN(H) = \bigcup_{h \in H} \cN(h) \setminus H$ to be the neighborhood of $H \subseteq V$.
 
A subset $H\sse V$ is called a \emph{clique} (resp. \emph{independent set}) if $\{u,v\}\in E$ (resp.\ $\{u,v\}\not\in E$) for all $u,v\in H$ with $u\neq v$.
A graph $\Gamma = (V,E)$ is called \emph{complete} if $V$ is a clique; it is called a \emph{cycle} if the entire graph is a cycle as defined above.

\subsection{Geodetic graphs}

\begin{definition}
A graph is \emph{geodetic} if each pair of vertices is connected by a unique geodesic.
\end{definition}
The following equivalent definition using even cycles is due to Stemple and Watkins.

\begin{lemma}[{\cite[Theorem~2]{stemple1968planar}}]
\label{lem:even_cycle}
  A connected graph $\Gamma$ is geodetic if and only if $\Gamma$ contains no even cycle $x_0, \dots, x_{2n} = x_0$ with $n\geq 2$ such that $d(x_i, x_{i+n}) = n$ for all $0 \le i \le n$.
\end{lemma}

An important special case are cycles of length four or six, the conditions under which those can exist are described in~\cite{stemple1974geodetic}.
Of these we recall some basic facts on 4-cycles. 

\begin{lemma}[{\cite[Theorem~3.3]{stemple1974geodetic}}]
  \label{lem:diamonds}
  Suppose that $\Gamma = (V, E)$ is a geodetic graph.
  Then the vertices of every cycle $v_0, v_1, v_2, v_3, v_0$ of length four in $\Gamma$ induce a complete subgraph.
\end{lemma}

\begin{remark}
    Throughout, we use \cref{lem:diamonds} in the following way without giving further reference.
    If $\Sigma$ is a generating set of a group $G$ such that the corresponding Cayley graph is geodetic and $a,b,c,d \in \Sigma$ with $a \neq c$ and $ab = cd\neq 1$, then $ab, cd, c^{-1}a, bd^{-1} \in \Sigma$ (since $1, a, ab,c, 1$ is a cycle of length $4$).
\end{remark}

\begin{lemma}[{\cite[Theorem~3.5]{stemple1974geodetic}}]
  \label{lem:adjacent_part_of_clique}
  Let $\Gamma = (V, E)$ be a geodetic graph and $C \subseteq V$ be a clique in $\Gamma$.
  If $v \in V$ is adjacent to at least two distinct vertices in $C$, then $C \cup \{ v \}$ is a clique.
\end{lemma}

\begin{lemma}
  \label{lem:partition_neighbors_cliques}
  Let $\Gamma = (V, E)$ be a geodetic graph.
  Then the neighbors of any vertex $v \in V$ can be partitioned into a set of disjoint cliques.
\end{lemma}

\begin{proof}
  Assume for contradiction that $x,y,z \in \mathcal{N}(v)$ with $\{x,y\} \in E$, $\{y,z\} \in E$ but $\{x,z\} \notin E$.
  Then there are two geodesics from $x$ to $z$ (one via $v$ and one via $y$), contradicting $\Gamma$ being geodetic. 
\end{proof}

After fixing a starting point in a Cayley graph $\Cay(G, \Sigma)$ (for instance the vertex corresponding to the identity element $1$), paths starting at $1$ are in bijection with words $w \in \Sigma^*$ (recall that we always assume the generating set $\Sigma \sse G$ to be symmetric). This allows us to denote paths by words in $\Sigma^*$ rather than vertex sequences and leads to the following observation.
\begin{observation}
	The Cayley graph $\Cay(G,\Sigma)$ of a group $G$ is geodetic if and only if each 
    element $g \in G$ is represented by a unique geodesic $w \in \Sigma^\ast$, with the identity element  represented by the empty word $\varepsilon\in \Sigma^\ast$.
\end{observation}

\section{Structure of geodetic Cayley graphs}\label{sec:theoryProofs}

\subsection{Complete subgroups}

\begin{definition}
    Let $G$ be a group with generating set $\Sigma$.
    We call $H \le G$ a \emph{complete subgroup} (with respect to $\Sigma$) if $H \setminus \{1\} \sse \Sigma$ or, equivalently, if $H \sse G$ induces a complete subgraph of $\Cay(G,\Sigma)$.
\end{definition}

\begin{lemma}\label{lem:cliques_subgroup_or_partner}
    Let $\Gamma = \mathrm{Cay}(G, \Sigma)$ be geodetic and let $C \sse G$ be a maximal clique of $\Gamma$ with $1 \in C$.

    \begin{enumerate}
    \item Then $g^{-1}C = C$ or $g^{-1}C \cap C =\{1\}$ for each $g \in C$.
    \item If $C$ is the only maximal clique $C'$ with $1 \in C'$ and $\abs{C'}=\abs{C}$, then $C$ is closed under inversion.
    \end{enumerate}
\end{lemma}

\begin{proof}
    Recall that multiplying by a group element on the left induces an isomorphism.
    Since $1 \in g^{-1}C \cap C$ for each $g \in C$, the first statement follows from \Cref{lem:adjacent_part_of_clique}.
    For the second statement, note that we have shown that $g^{-1}C = C$ holds for each $g \in C$ by uniqueness of $C$.
    As such, $g^{-1} \in g^{-1} C = C$.
\end{proof}

Note that, in the second case of the previous lemma, where $C$ is the only maximal clique of its size, $C$ is not only closed under inversion, it is also a subgroup.
We prove a more general statement.

\begin{lemma}\label{lem:symmetric_clique_subgroup}
  Let $\Gamma = \mathrm{Cay}(G, \Sigma)$ be geodetic.
  If $X \sse G$ is a clique of $\Gamma$ with $1 \in X$ and such that $X$ is closed under inversion, then $\langle X \rangle$ is a complete subgroup.
\end{lemma}

\begin{proof}
    Let $C \subseteq G$ be a maximal clique containing $X$. 
    If $g \in X\setminus\{1\}$, then $\{g^{-1}, 1, g\} \in C$ so $\{g^{-1}, 1\} \in g^{-1}C$.
    Hence $g^{-1}C\cap C\neq \{1\}$ and so $g^{-1}C=C$ by \Cref{lem:cliques_subgroup_or_partner}.
    Thus for any $g\in X$ we have $C = g(g^{-1}C) = gC$.
    By induction, assume products of $i$ elements of $X$ lie in $C$ which holds for $i=1$. Then $g_1\dots g_{i+1}\in g_1C=C$. This shows that $\gen{X}\subseteq C$ so $\gen{X}$ is a clique.
\end{proof}

We observe that the second item of \Cref{lem:cliques_subgroup_or_partner} and \Cref{lem:symmetric_clique_subgroup} together imply \cite[Lemma~14]{wrap108882} of Federici.

\begin{lemma}\label{lem:complete_subgroup_join}
  Let $\mathrm{Cay}(G, \Sigma)$ be geodetic and let $H_1, H_2 \leq G$ be complete subgroups. 
  If $H_1 \cap H_2 \neq \Triv$ then $\langle H_1, H_2 \rangle$ is also a complete subgroup of $G$.
\end{lemma}

\begin{proof}
  Under these assumptions $H_1 \cup H_2$ is a clique (\Cref{lem:adjacent_part_of_clique}).
  Since $1 \in H_1 \cup H_2$ and $H_1 \cup H_2$ is closed under inversion, $\langle H_1, H_2 \rangle$ is a clique by \Cref{lem:symmetric_clique_subgroup}.
\end{proof}

\begin{lemma}
  \label{lem:clique_large_independentset}
  Let $\Gamma = (V, E)$ be a biconnected vertex-transitive non-complete geodetic graph and let $C \sse V$ be a clique of $\Gamma$ of size $k$.
  Then there is an independent set $I \sse \cN(C)$ 
  of size $k^2 - k$.
\end{lemma}

\begin{proof}
  By~\cite[Corollary 1]{gorovoy2021graphs} every biconnected non-complete geodetic graph containing a clique of size $k$ also contains the star graph $K_{1,k}$ as an induced subgraph.
  In other words, there exists a vertex $v \in V$ with an independent set of size $k$ in its neighborhood $\cN(v)$.
  As $\Gamma$ is vertex-transitive, for every vertex $v \in V$ there is an independent set $I(v) \sse \cN(v)$ of size $k$.

  Fix a maximal clique $\tilde C$ of $\Gamma$ with $C \sse \tilde C$ and set $\tilde I(v) \coloneqq I(v) \setminus \tilde C$. 
  Note that $\tilde I(v)\subseteq \cN(C)$. 
  Since an independent set can contain at most one vertex of any given clique, we have $\abs{\smash{\tilde{I}}(v)} \ge \abs{I(v)} - 1 = k - 1$.

  We claim that $\tilde{I}(u) \cap \tilde{I}(v) = \emptyset$ for distinct $u, v \in C$.
  If there were a vertex $x \in \tilde{I}(u) \cap \tilde{I}(v)$, then we would have $x \in \tilde{C}$ by \Cref{lem:adjacent_part_of_clique}, because $x$ is adjacent to two vertices in $\tilde{C}$.
  That contradicts the definition of $\tilde{I}$, which excludes vertices in $\tilde{C}$.

  Let $I \coloneqq \bigcup_{v \in C} \tilde{I}(v) \sse \cN(C)$ 
  and observe that $\abs{I} \geq k(k-1) = k^2 - k$, as the union is disjoint.
  It remains to show that $I$ is indeed an independent set.
  Assume there are $x, y \in I$ with $\{x, y\} \in E$.
  By definition of $I$ there are $u, v \in C$ such that $x \in \tilde{I}(u)$ and $y \in \tilde{I}(v)$.
  We must have $u \neq v$, as $\tilde{I}(u)$ is an independent set.
  Now we have a 4-cycle $u, x, y, v$.
  By \Cref{lem:diamonds} this implies the existence of the edge $\{u, y\}$.
  Now $y$ is a neighbor of both $u$ and $v$.
  That implies $y \in \tilde{C}$ by \Cref{lem:adjacent_part_of_clique}, a contradiction with the definition of $\tilde{I}(v)$.
\end{proof}

\begin{proposition}\label{pro:complete_subgroup_bound}
    Let $\Gamma = \mathrm{Cay}(G, \Sigma)$ be a geodetic but not complete Cayley graph.
    If $H \leq G$ is a complete subgroup with respect to $\Sigma$, then $\abs{G} \ge \abs{H}^3 - \abs{H}^2 + \abs{H}$.
\end{proposition}

\begin{proof}
    By \Cref{lem:clique_large_independentset} there is an independent set of size $\abs{H}^2 - \abs{H}$ in the neighborhood of $H$ which does not contain any vertex of $H$.
    We now look at the left cosets of $H$.
    Each such coset is a clique of $\Gamma$.
    Thus, each coset can contain at most one point of the independent set.
    Since $H$ itself does not contain any point of the independent set, the index of $H$ is at least $\abs{H}^2 - \abs{H} + 1$.
    Hence, $\abs{G} = \abs{G : H} \abs{H} \ge \abs{H}^3 - \abs{H}^2 + \abs{H}$.
\end{proof}

\subsection{Conjugacy classes and elements of order two}

\begin{lemma}[Conjugacy class of generators]\label{lem:conjugacyclass}
    Let $G$ be a group with geodetic Cayley graph $\mathrm{Cay}(G, \Sigma)$ and $H \leq G$ such that $H = \langle H \cap \Sigma \rangle$.
    If $x^H \subseteq \Sigma$ for some $x \in \Sigma$, then $H \cap \Sigma = \{ x^{\pm1} \}$ or $H$  is a complete subgroup.
    In particular, if $x^G \sse \Sigma$ for some $x \in G$, then $\Cay(G, \Sigma)$ satisfies \cref{conj:main}.
\end{lemma}
\begin{proof}
    If $y \in H \cap \Sigma$ with $y \neq x^{\pm1}$, then $x^y \in \Sigma$ and, therefore, $yx^y = xy \in \Sigma$ as this element cannot have length zero or two.
    We also have $x^{-1}y \in \Sigma$ and $y x^{\pm1} \in \Sigma$ by symmetry.
    
    If $H\cap \Sigma\neq \{x^{\pm1}\}$, then there exists $y \in H \cap \Sigma$ with $y \neq x^{\pm1}$.
    We will show that every $uv \in G$ with $u,v \in H \cap \Sigma$ has length at most one.
    If $u \not\in \{x^{\pm1}\}$ and $v \in \{x^{\pm1}\}$, or vice versa, then this follows from the argument above.
    Otherwise, first consider the case that  $u,v \not\in \{x^{\pm1}\}$ and define $u' \coloneqq u x$ and $v' \coloneqq \overline{x}v$.
    As $u' = ux = x^{\overline{u}}u$, we have $u'\in \Sigma$ and likewise $v' \in \Sigma$.
    Since $uv=u'v'$, it follows that $uv \in \Sigma$.
    If $u,v \in \{x^{\pm1}\}$, we set $u' \coloneqq uy$ and $v' \coloneqq y^{-1}v$ and by the same argument conclude that $uv \in \Sigma$.
    
    Finally, if $x^G \subseteq \Sigma$, then either $\Sigma = G \cap \Sigma = \{x^{\pm1}\}$ or $G$ is a complete subgroup of itself (or both in case $G \cong C_2$ or $G \cong C_3$).
    In the first case, $G = \gen{x^{\pm 1}}$ is a cyclic group.
    If, moreover, $\abs{G} > 2$, then $\Cay(G, \Sigma)$ is a cycle of length $\abs{G}$, which has to be odd as an even cycle is not geodetic.
    Otherwise as well as in the second case, $\Cay(G, \Sigma)$ is complete.
\end{proof}

The above lemma may seem technical but is extremely useful, for example, as demonstrated by the following consequences.
For the first of these, note that $x^G = \{x\}$ whenever $x \in Z(G)$.

\begin{corollary}[Generator in center]\label{cor:center}
    If $Z(G) \cap \Sigma \neq \emptyset$, then $\Cay(G, \Sigma)$ satisfies \cref{conj:main}.
\end{corollary}

From this we obtain the following result first shown by Georgakopoulos and presented in \cite{wrap108882}.

\begin{corollary}[{\cite[Proposition 10]{wrap108882}}]\label{cor:abelian}
    If $G$ is a finite abelian group, then $G$ satisfies \Cref{conj:main}.
\end{corollary}
\begin{proof}
    This is immediate from \cref{cor:center} since $G=Z(G)$.
\end{proof}
Note that \cref{lem:conjugacyclass} can also be applied to deduce the completeness of a subgroup. 
A simple example of such an application is as follows.
(Another example of this kind can be found in \cref{lem:saturate_c6_d6} and both of these observations are used to facilitate our computer search.)

\begin{corollary}[Commuting generators]\label{cor:commgen_subgroup}
    Let $\mathrm{Cay}(G, \Sigma)$ be geodetic and $x,y \in \Sigma$.
    If $xy = yx$ and $y\neq x^{\pm1}$, then $\langle x, y \rangle \leq G$ is a complete subgroup.
\end{corollary}
\begin{proof}
    Apply \Cref{lem:conjugacyclass} with $H = \langle x, y \rangle$.
\end{proof}

If a group has an element of order two we have the following.
\begin{lemma}[Order-two elements]\label{lem:conjugate_ord_two}
    Let $G$ be a group with generating set $\Sigma$ such that the Cayley graph is geodetic.
    Let $g \in G$ be an element of order two.
    Then the geodesic for $g$ is of the form 
    \[
        (w_1 \cdots w_{\ell}) \cdot w_{\ell + 1} \cdot (\overline{w_{\ell}} \cdots \overline{w_1})
    \]
    of length $2\ell+1$ where $\ell\in\mathbb N$ and $w_i\in\Sigma$ for each $1 \leq i \leq \ell +1$. 
    In particular, $\Sigma$ contains an element of order two conjugate to $g$, namely $w_{\ell + 1}$. 
\end{lemma}

\begin{proof}
  Let $w_1 \cdots w_k \in \Sigma^\ast$ be a geodesic for $g$.
  We have $g = g^{-1}$ and thus $w_1 \cdots w_k = \overline{w_k} \cdots \overline{w_1}$.
  Hence, we have two paths of length $k$ that lead from $1$ to $g$.
  Since $\Cay(G, \Sigma)$ is geodetic, $g$ must have a unique geodesic, \ie the two paths must coincide; hence, $w_i = \overline{w_{k-i+1}}$ for $i \in \{1, \dots, k\}$.
  
  If $k$ is even, then $g = (w_1 \cdots w_{k/2}) \cdot (w_{k/2 + 1} \cdots w_k) \equiv (w_1 \cdots w_{k/2}) \cdot (\overline{w_{k/2}} \cdots \overline{w_1}) = 1$, contradicting the assumption that $g$ is of order two.
  Thus, $k = 2 \ell + 1$ must be odd and we obtain the equation
  \begin{align*}
    g &= (w_1 \cdots w_{\ell}) \cdot w_{\ell + 1} \cdot (w_{\ell + 2} \cdots w_k) \\
    &\equiv (w_1 \cdots w_{\ell}) \cdot w_{\ell + 1} \cdot (\overline{w_{\ell}} \cdots \overline{w_1})\text{.}
  \end{align*}
  Hence, $g$ is conjugate to $w_{\ell + 1} \in \Sigma$ and $w_{\ell + 1} = \overline{w_{\ell + 1}}$, proving the lemma.
\end{proof}

The above results lead to the following observation which turns out to be extremely useful in any computer search, since it shows that every $2$-group immediately satisfies \Cref{conj:main}.

\begin{theorem}[Even order center]
  \label{cor:evencenter}
  Let $G$ be a group such that $Z(G)$ is of even order.
  Then the only geodetic Cayley graph of $G$ is the complete graph.
\end{theorem}

\begin{proof}
  Assume we have a geodetic Cayley graph of $G$.
  By \Cref{lem:conjugate_ord_two}, the generating set $\Sigma$ must contain at least one element of each conjugacy class of elements of order two in $G$.
  Since $Z(G)$ is of even order, it must contain an element $g$ of order two.
  As $g$ is in the center of the group, it is not conjugate to any other elements.
  Hence, $g \in \Sigma$ and by \Cref{cor:center}, the Cayley graph is either complete, or an odd cycle.
  Since $G$ contains an even order subgroup, $G$ itself has even order.
  Therefore, the Cayley graph cannot be an odd cycle, so it must a be complete graph.
\end{proof}

\begin{remark}
    All but $4734$ groups of the approximately $49.5$ billion groups of order at most $1024$ are covered by the combination of \Cref{cor:evencenter} and \Cref{cor:abelian}.
\end{remark}

\subsection{Cayley graphs of diameter two}

\begin{proposition}\label{prop:notCG}
  If $G$ is of even order, then $G$ has no geodetic Cayley graph of diameter two.
\end{proposition}
\begin{proof}
Assume that $\Gamma = \mathrm{Cay}(G, \Sigma)$ is a geodetic Cayley graph of diameter two and $\abs{G}$ is even.
 Since $\abs{G}$ is even, $G$ has elements of order two.
  Since $\Gamma$ is geodetic and has diameter two, each such element is contained in $\Sigma$ by \Cref{lem:conjugate_ord_two}.
  But then $\Sigma$ contains an entire conjugacy class.
  By \Cref{lem:conjugacyclass}, the graph $\Gamma$ would therefore have to be complete or a cycle, both of which are absurd.
\end{proof}

This affords a short elementary proof of the (well-known) fact that the Moore graphs of diameter two, other than $C_5$, are not Cayley graphs (see \cite[Theorem~3.13]{Cameron99}, \cite{PotonikSV2012,MillerS05}).

\begin{corollary}\label{cor:notCG}
  The Petersen graph, the Hoffman-Singleton graph, and all of the hypothetical Moore graphs of degree $57$ and diameter two are not Cayley graphs.
\end{corollary}
\begin{proof}
As each graph has an even number of vertices and diameter two, this follows from \cref{prop:notCG}.
\end{proof}

In order to prove the next lemma, we recall the following definition.
A graph $\Gamma$ is called \emph{strongly regular} with parameters $(\delta,\lambda,\mu)$ if every vertex has degree $\delta$, any two adjacent vertices share exactly $\lambda$ neighbors, and any two non-adjacent vertices have exactly $\mu$ neighbors in common. A necessary condition for a graph with $v$ vertices to be strongly regular with parameters $(\delta,\lambda,\mu)$ is the equation $
(v-\delta-1)\mu =\delta(\delta-\lambda -1)$.
Clearly, every strongly regular graph with parameter $\mu = 1$ is geodetic and has diameter two, and in this case 
the above condition becomes 
\begin{equation}
    v = \delta(\delta - \lambda) + 1. \label{eq:diam_to_vertex_count}
\end{equation}

\begin{lemma}
  \label{lem:small_order_diameter_two_implies_c5}
  If $\Gamma = \mathrm{Cay}(G, \Sigma)$ is geodetic and has diameter two and $\abs{G} < 2025$, then $\Gamma$ is the cycle $C_5$.
\end{lemma}

\begin{proof}
From \cite{stemple1974geodetic,Kantor77} (see also \cite[Theorem 1]{blokhuis1988geodetic}), we have that every geodetic graph of diameter two falls into one of the following classes: block graphs joining all cliques in one vertex, biconnected graphs with exactly two different vertex degrees, and strongly regular graphs with parameter $\mu = 1$.

Among these, the strongly regular graphs are the only regular graphs.
If the parameter $\lambda$ is zero, then $\Gamma$ is a Moore graph.
The Moore graphs of diameter two are the cycle $C_5$, the Petersen graph, the Hoffman-Singleton graph and hypothetical graphs of degree $57$~\cite{BannaiI73,Damerell73}.
Hence, by \cref{cor:notCG}, only $C_5$ remains.

Deutsch and Fischer~\cite{DeutschF01} showed that if $\lambda > 0$, then we have $\lambda > 1$ and either $(\delta,\lambda) = (21,2)$ or 
\begin{equation}
  \delta \geq (\lambda + 1) (\lambda + 13)
  \label{eq:DeutschFischer}
\end{equation}
\cite[Theorem 4.1 and Corollary]{DeutschF01}.
By the Handshake Lemma, an odd degree is only possible for a group of even order and hence is excluded by \cref{prop:notCG}. This excludes the case $(21, 2)$.
For the second case, if $\lambda \geq 3$, then Equations~\eqref{eq:diam_to_vertex_count} and~\eqref{eq:DeutschFischer} yield $\abs{G} = \delta (\delta - \lambda) +1\geq \lambda^4 + 27 \lambda^3 + 208 \lambda^2 + 351 \lambda + 170 \geq 3905$.
For $\lambda = 2$ we obtain $\delta \geq 45$. 
Again, $\delta$ must be even; hence, $\delta \geq 46$ and  $\delta - \lambda \geq 44$.
Applying these bounds to Equation~\eqref{eq:diam_to_vertex_count}, we 
obtain $\abs{G} = \delta(\delta-\lambda)+1 \geq 46\cdot 44 + 1=2025$.
\end{proof}

We note that an alternative proof of 
\cref{lem:small_order_diameter_two_implies_c5} with a bound of 1300 has been provided to us by Filippo Prandina, which relies on a database by Brouwer\footnote{ \url{https://aeb.win.tue.nl/graphs/srg/srgtab.html}} of strongly regular graphs.

\subsection{Central elements}\label{subsec:center}

\begin{lemma}
  \label{lem:commgen_square}
  Let $\mathrm{Cay}(G, \Sigma)$ be geodetic and $b, t \in \Sigma$.
  Suppose that $b^2 \neq 1$, $t^2 = 1$, and $b^2t = tb^{\pm2}$.
  Then the subgroup $\langle b, t \rangle \le G$ is complete with respect to $\Sigma$.
\end{lemma}

\begin{proof}
  If $bt = tb^{\pm1}$, then the statement follows from \Cref{lem:conjugacyclass}, so assume that $bt \neq tb^{\pm1}$ holds.
  As $bbt = tbb$ or $bbt = t\overline b \overline b$, we have $s \coloneqq b t \overline b = \overline b t b$ or $s \coloneqq btb = \overline b t \overline b$, respectively.
  In particular, there are two words of length three representing $s$; thus there must be a shorter one.
  Moreover, since $s$ has order two, its geodesic must have odd length.
  Hence, $s \in \Sigma$.
  The situation is as shown in \Cref{fig:commgen_square}.
  
\begin{figure}[h]
  \centering
  \begin{center}
    \begin{tikzpicture}[rotate=-45, scale=1.2]
      \foreach \x/\y in {0/0, 0/1.25, 0/2.5, 1.25/0, 1.25/1.25, 1.25/2.5, 1.25/3.75, 2.5/1.25, 2.5/2.5, 2.5/3.75}
        \fill (\x,\y) circle (1.25pt);
      \begin{scope}[shorten <= 3pt, shorten >= 3pt]
        \foreach \x/\y/\l in {0/0/t, 0/1.25/s, 0/2.5/t, 1.25/1.25/t, 1.25/2.5/s, 1.25/3.75/t}
            \draw (\x,\y) -- node[below left, inner sep=1pt] {$\l$} +(1.25,0);
      \end{scope}
      \begin{scope}[shorten <= 3pt, shorten >= 3pt, ->]
        \foreach \x/\y in {0/0, 0/1.25, 1.25/0, 1.25/1.25, 1.25/2.5, 2.5/1.25, 2.5/2.5}
            \draw (\x,\y) -- node[above left, inner sep=1pt] {$b$} +(0,1.25);
      \end{scope}
    \end{tikzpicture}
    \hfil
    \begin{tikzpicture}[rotate=-45, scale=1.2]
      \foreach \x/\y in {0/0, 0/1.25, 0/2.5, 1.25/0, 1.25/1.25, 1.25/2.5, 1.25/3.75, 2.5/1.25, 2.5/2.5, 2.5/3.75}
        \fill (\x,\y) circle (1.25pt);
      \begin{scope}[shorten <= 3pt, shorten >= 3pt]
        \foreach \x/\y/\l in {0/0/t, 0/1.25/s, 0/2.5/t, 1.25/1.25/t, 1.25/2.5/s, 1.25/3.75/t}
            \draw (\x,\y) -- node[below left, inner sep=1pt] {$\l$} +(1.25,0);
      \end{scope}
      \begin{scope}[shorten <= 3pt, shorten >= 3pt]
        \foreach \x/\y/\d in {0/0/->, 0/1.25/->, 1.25/0/<-, 1.25/1.25/<-, 1.25/2.5/<-, 2.5/1.25/->, 2.5/2.5/->}
            \draw[\d] (\x,\y) -- node[above left, inner sep=1pt] {$b$} +(0,1.25);
      \end{scope}
    \end{tikzpicture}
  \end{center}
    \caption{The situation in the proof of \Cref{lem:commgen_square}. 
    Note that $s,t$ have order two, so edges labeled $s,t$  are drawn as undirected.}
  \label{fig:commgen_square}
\end{figure}

  Let $H \coloneqq \langle b, t, s \rangle =\gen{H\cap \Sigma}\leq G$.
  One may verify by straightforward computations (conjugating $st$ by $b,\bar b, s, t$) that $(st)^H = \{st, ts\} \sse \Sigma$.
  The statement then follows using \Cref{lem:conjugacyclass}.
\end{proof}

\begin{lemma}
  \label{lem:commgen_square_pt2}
  Let $\mathrm{Cay}(G, \Sigma)$ be geodetic.
  If there exists some $b \in \Sigma$ with $b^2 \neq 1$ and $(b^2)^G \sse \{ b^{\pm2} \}$, then $\mathrm{Cay}(G, \Sigma)$ is an odd cycle or a complete graph.
\end{lemma}

\begin{proof}
  Suppose first, that $(b^2)^G = \{ b^2 \}$ and thus $b^2 \in Z(G)$.
  Then either $b \in Z(G)$ or $b \not\in Z(G)$.
  In the first case, the statement follows immediately from \Cref{cor:center}.
  In the second case, $\abs{G : Z(G)}$ is even; hence so is $\abs{G}$.
  As such, there exists an element of order two in $G$.
  In fact, there even exists $t \in \Sigma$ with $t^2 = 1$ by \Cref{lem:conjugate_ord_two}.
  We then obtain $(b^2)^G \subseteq \Sigma$ from \Cref{lem:commgen_square}.
  The statement now follows from \Cref{cor:center}.
  
  Finally, suppose that $(b^2)^G = \{ b^{\pm2} \}$ has two elements.
  Then $\abs{G : C_G(b^2)} = 2$ and, as before, there thus exists some $t \in \Sigma$ with $t^2 = 1$. 
  The statement follows from \Cref{lem:commgen_square} and \Cref{lem:conjugacyclass}.
\end{proof}

\begin{lemma}[Flat coset]
  \label{lem:commgen_coset}
  Let $G$ be a group with generating set $\Sigma$ such that the Cayley graph is geodetic.
  If  $w \in \Sigma^{m}$ is a geodesic of length $m$ and $a \in \Sigma$ such that $aw = wa$ and $w \neq a^{\pm m}$, 
  then each $ h \in w \langle a \rangle$ has a geodesic of length at most $m$.
  Moreover, if $ w \not\in  \langle a \rangle$, then each $ h \in w \langle a \rangle$ has a geodesic of length exactly $m$.
\end{lemma}

\begin{proof}
  We denote by $w_i$ the geodesic of $w a^i$.
  Let $r$ be the order of $a$ in $G$.
  Observe that $w_0 = w_r$.
  
  First, we prove the lemma for the case $w \not\in  \langle a \rangle$.
  In this case, $a w_i$ and $w_i a$ are two different words, as otherwise $w \in \langle a \rangle$.
  Both $a w_i$ and $w_i a$ are of length $|w_i| + 1$ representing $w a^{i+1}$.
  Thus, there must be a shorter word, implying that $| w_{i+1}| \le |w_i|$.
  Therefore, $m = |w_r| \le \dots \le |w_i| \le \dots \le |w_0| = m$, and we conclude $|w_i| = m$ for all $i$.
  
  Second, we consider the case $w = a^k$.
  The proof of this case is illustrated in \Cref{fig:coset_wg}.
  Recall that $w_i$ is the geodesic for $wa^i = a^{i+k}$.
  Observe that, because the geodesic of $a^k$ has length $m$, we must have $m \le k$.
  In fact, because $w \neq a^{\pm m}$, we have $m < k$.
  Using the same argument, we also have $m < r - k$.
  We prove by induction on $i$ that $|w_i| \le m$.
  For $w_0 = w_r = w$ this obviously holds.

  If $0 < i \le r-k$, then we have two words, $w_{i-1} a$ and $a w_{i-1}$ of length $|w_{i-1}| + 1$ for $a^{i+k} = w_i$.
  To prove that the two words are different, we show $w_{i-1} \neq a^{|w_{i-1}|}$:
  If $w_{i-1}$ were a power of $a$, then $w_{i-1} = a^{i + k - 1}$.
  That is impossible, as $|w_{i-1}| \le m < k \le i + k - 1$.
  Thus, the geodesic $w_{i}$ has length at most $|w_{i-1}| \leq m$.

  For $r-k < i < r$, we also use induction on decreasing values of $i$, multiplying with $\overline{a}$.
  Now $w_r = w$ is the base case for the induction.
  Assuming the statement holds for $i+1$, the two words $w_{i+1} \overline{a}$ and $\overline{a} w_{i+1}$ both correspond to the group element $a^{i + k} = w_i$.
  To prove that the two words are different, we show $w_{i+1} \neq {\ov a}^{|w_{i+1}|}$:
  If $w_{i+1}$ were a power of $\ov a$, then $w_{i+1} = {\ov a}^{2r- k - i - 1}$ leading to the contradiction $2r - k - i - 1 \ge r - k > m \ge |w_{i+1}|$.
  Therefore, we must have $|w_i| \leq |w_{i+1}| \leq m$.
\end{proof}

\begin{figure}[h]
  \centering
  \begin{tikzpicture}[scale=1.2]
    \foreach \a/\n in {90, 126, ..., 414} {
        \fill (\a:2) circle (1.25pt);
    }

    \begin{scope}[shorten <=3pt, shorten >=3pt, ->]
        \draw (90:2) 
            arc[start angle=90, end angle=126, radius=2]
            node[pos=.6, below right] {$\bar a\mathstrut$};
        \draw (162:2) 
            arc[start angle=162, end angle=198, radius=2]
            node[pos=.5, right] {$\bar a\mathstrut$};
        \draw (234:2) 
            arc[start angle=234, end angle=270, radius=2]
            node[pos=.4, above right] {$\bar a\mathstrut$};

        \draw (90:2) 
            arc[start angle=90, end angle=54, radius=2]
            node[pos=.6, below left] {$a\mathstrut$};
        \draw (18:2) 
            arc[start angle=18, end angle=-18, radius=2]
            node[pos=.5, left] {$a\mathstrut$};
        \draw (-54:2) 
            arc[start angle=-54, end angle=-90, radius=2]
            node[pos=.4, above left] {$a\mathstrut$};
    \end{scope}

    \begin{scope}[shorten <=3.5pt, shorten >=3pt, <-, 
        dash pattern=on 9.5pt off 2pt on 3pt off 2pt on 3pt off 2pt on 3pt off 2pt on 3pt off 2pt on 3pt]
        \foreach \ea/\sa in {126/162, 198/234, 54/18, -18/-54}
            \draw (\sa:2) arc[start angle=\sa, end angle=\ea, radius=2];
    \end{scope}

    \node[anchor=south, shift={(0,+.15)}] at (+90:2) 
        {$w_0 = a^k$};
    \node[anchor=north, shift={(0,-.15)}] at (-90:2) 
        {$w_{r-k} = 1$};

    \node[anchor=east, shift={(-.13,0)}] at (126:2) 
        {$w_{r-1} = a^{k-1}$};
    \node[anchor=east, shift={(-.13,0)}] at (162:2) 
        {$w_{r-k+d+1} = a^{d+1}$};
    \node[anchor=east, shift={(-.20,0)}] at (198:2) 
        {$w_{r-k+d} \equiv a^d$};
    \node[anchor=east, shift={(-.27,0)}] at (234:2) 
        {$w_{r-k+1} \equiv a^1$};

    \node[anchor=west, shift={(+.29,0)}] at ( 54:2) 
        {$w_1 = a^{k+1}$};
    \node[anchor=west, shift={(+.22,0)}] at ( 18:2) 
        {$w_{r-k-d-1} = \bar a^{d+1}$};
    \node[anchor=west, shift={(+.22,0)}] at (-18:2) 
        {$ w_{r-k-d} \equiv \bar a^d$};
    \node[anchor=west, shift={(+.29,0)}] at (-54:2) 
        {$w_{r-k-1} \equiv \bar a^1$};
  \end{tikzpicture} 

  \caption{The coset $w\langle a \rangle$ in the case $w = a^k$ in the proof of \Cref{lem:commgen_coset}.}
  \label{fig:coset_wg}
\end{figure}

The second case of the proof, where $w = a^k$ is illustrated in \Cref{fig:coset_wg}.
One can observe the two inductions, starting at $w$, one multiplying with $a$, the other multiplying with $\overline{a}$ covering the entire subgroup.
Before arriving at $1$ we encounter small powers of $a$, respectively $\overline{a}$.
In fact, there is an integer $d \le |w|$, such that $a^i$ and $\overline{a}^i$ are geodesic for $0 \le i \le d$, but for $j > d$ the words $a^j$ and ${\ov a}^j$ are not geodesics.
By applying \Cref{lem:commgen_coset} to the geodesic of $a^{d+1}$, which has length $d$, we obtain $d = m$. 
This proves the following.

\begin{corollary}
    Let $\Cay(G, \Sigma)$ be a counterexample to \Cref{conj:main} and let $a \in \Sigma$.
    Then there is an integer $m \in \mathbb{N}$ such that $a^{\pm i}$ is a geodesic if $i \le m$ and each element in the set $\langle a \rangle \setminus \{ a^{\pm i} \mid 0 \le i \le m\}$ has a geodesic of length exactly $m$.
\end{corollary}

\begin{lemma}[Geodesics for central elements are powers of generators]
  \label{cor:centerelementispower}
  Let $G$ be a group with geodetic Cayley graph $\Cay(G, \Sigma)$, and let $z \in Z(G) \setminus \{1\}$.
  Then the unique geodesic of $z$ is $a^k$ for some $a \in \Sigma$ and~$k \ge 1$.
\end{lemma}

\begin{proof}
  Let $w_1\cdots w_k \in \Sigma^*$ be a geodesic for $z$.
  As $z$ is in the center we have
  \begin{equation*}
    w_1 w_2 \cdots w_k = z = z^{w_1}  = w_2 \cdots w_k w_1\text{.}
  \end{equation*}
  Now we have two words of length $k$ representing $z$.
  There cannot be a shorter word, as we chose $w_1\cdots w_k$ to be a geodesic.
  The Cayley graph is geodetic, so the words must coincide:
  we obtain $w_1 = w_2 = \dots = w_k$.
\end{proof}

\begin{theorem}[Geodesics for central elements are powers with the same exponent] \label{thm:center_geodesics}
  Let $G \neq \Triv$
  be a group with geodetic Cayley graph $\Cay(G,\Sigma)$ and $\Sigma \cap Z(G) = \emptyset$.
  Then, there is an integer $k \ge 3$ such that for each $z \in Z(G) \setminus \{1\}$ there exists an element $c_z \in \Sigma$ such that $c_z^k$ is the geodesic for $z$.
  Moreover, $k$ divides $\abs{G : Z(G)}$.
\end{theorem}

In light of this theorem, we make the following definitions for a non-cyclic geodetic Cayley graph $\Cay(G, \Sigma)$ with $G \neq \Triv$ and $\Sigma \cap Z(G) = \emptyset$. If $k\geq 3, z\in Z(G) \setminus \{1\}$ and $c_z\in \Sigma$ are as in  \cref{thm:center_geodesics},
\begin{enumerate}
    \item $c_z$ is called a \emph{central root}
    \item $k$ is called the \emph{length of central geodesics}
    \item a geodesic between two central elements is called a \emph{central geodesic}, and is necessarily labeled by $c_z^k$ for some central root $c_z$.
\end{enumerate}
Observe that if $\Sigma \cap Z(G) \neq \emptyset$, then $k = 1$, as in this case the Cayley graph is complete by \Cref{cor:center}.

\begin{proof}
  We know by \Cref{cor:centerelementispower} that the geodesic of every element in the center is a power of a generator.
  We first show that no two central elements can have a geodesic that is a power of the same generator (note that a generator and its inverse is allowed and will happen).
  Let $b^k$ be a geodesic representing a central element and assume that $k$ is minimal so that $b^k$ is central.
  Let $g$ be a generator different to $b^{\pm 1}$ (such a generator must exist, otherwise $G$ would be a cyclic group, but then $b \in Z(G)$).
  By~\Cref{lem:commgen_coset} the elements $b^k g$ and $b^k \overline{g}$ both have a geodesic of length at most $k$, which we write as $(b^k g)$ and $(b^k \overline{g})$.
  Now $(b^k g)(b^k \overline{g})$ is a word of length at most $2k$ representing the central element $b^{2k}$.
  Thus, $b^{2k}$ is not a geodesic.
  Assume there is some $i \ge 1$ such that $b^{k+i}$ is a geodesic of a central element.
  This implies that $b^i$ is a geodesic of a central element.
  By minimality of $k$ we conclude $i\geq k$.
  It follows that no power of $b$ other than $b^k$ is a geodesic for a non-trivial central element.
  In particular, the order of $b \cdot Z(G)$ in $G / Z(G)$ is $k$ and, thus, $k$ divides $\abs{G : Z(G)}$.

  Next we show that the geodesics of all non-trivial central elements have the same length.
  For a central element and its inverse this is obvious.
  Assume there are $c, d \in \Sigma$ with $c \notin \{d, \bar{d}\}$ and $k,\ell \geq 1$ such that $c^k$ and $d^\ell$ both are geodesics representing different central elements.
  Note that $c^k \neq \bar{d}^{\ell}$ because both are geodesics.
  Now assume for a contradiction that $k \neq \ell$, w.\,l.\,o.\,g.\ $k < \ell$.
  \Cref{lem:commgen_coset} tells us that $c^k d^\ell \in c^k \langle d \rangle$ has a geodesic $u$ of length $\abs{u} = m \le k$.
  But $c^k d^\ell$ is also in the center (and non-trivial), so $u = f^m$ for some $f \in \Sigma$.
  Applying \Cref{lem:commgen_coset} again to $d^\ell = f^m c^{-k} \in f^m \langle c \rangle$, we obtain a geodesic $v$ of length at most $m \le k < \ell$, contradicting $d^\ell$ being a geodesic.

  It remains to show, that no central element has a geodesic of length two.
  In that case the Cayley graph would be cyclic or complete by \Cref{lem:commgen_square_pt2}, contradicting $\Sigma \cap Z(G) = \emptyset$. 
\end{proof}

From \cref{thm:center_geodesics} and \cref{cor:center}, we immediately obtain the following corollary.
\begin{corollary}\label{lem:center_no_diam_two}
If $\Gamma = \mathrm{Cay}(G, \Sigma)$ is geodetic and has diameter two and $Z(G) \neq 1$, then $\Gamma$ is the cycle $C_5$.
\end{corollary}

\begin{lemma}\label{lem:central_length_not_even}
    Let $\Cay(G, \Sigma)$ be geodetic.
    If there is an element $z \in Z(G)$, with $\ord(z) \ge 4$, then the length of central geodesics must be odd.
\end{lemma}

\begin{proof}
    Let $m = \ord(z)$.
    If $m$ is even, then the statement follows from \Cref{cor:evencenter}.
    Thus from here on we assume that $m$ is odd and $m \ge 5$.
    Assume that $b^{2k}$ is the geodesic for $z$.
    Let $y$ be the geodesic for $z^{(m-1)/2}$.
    By \Cref{thm:center_geodesics} it has length $2k$.
    As $m \ge 5$, we have $y \neq b^{\pm 2k}$.
    Note that $t = b^{km} = y b^{k}$.
    By \Cref{lem:commgen_coset}, the element $t \in y \langle b \rangle$ has length at most $|y| = 2k$.
    Clearly $t^2 = b^{2km} = 1$.
    If $t = 1$, then $y = tb^{-k} = \bar{b}^k$, contradicting $2k$ being the length of the central geodesics.
    It follows that $t$ has order two, thus by \Cref{lem:conjugate_ord_two} it must have odd length, that is, the length of $t$ is at most $2k - 1$.
    Applying \Cref{lem:commgen_coset} to $y = t \bar{b}^{k} \in t \langle b \rangle$ we obtain the contradiction $|y| \le |t| = 2k - 1$.
\end{proof}

\section{Bounds on the size of generating sets}\label{sec:bounds}

In this section we establish bounds on the possible sizes of those generating sets which result in a geodetic Cayley graph that is a counterexample to \cref{conj:main}.
To obtain these, we study the structure and size of balls of radius one and two in such a graph, as well as the positional relationships of central geodesics.

\smallbreak

Let $\Cay(G, \Sigma)$ be an arbitrary geodetic Cayley graph.
Throughout, we denote the $r$-ball in $\Cay(G, \Sigma)$ centered at a vertex $g \in G$ by $\cB_r(g) \coloneqq \{ h \in G \,\mid\, d(g,h) \le r \}$.
As Cayley graphs are vertex-transitive, all balls of the same radius $r$ are isomorphic subgraphs of $\Cay(G, \Sigma)$.

We begin by analyzing the structure of one-balls.
By \Cref{lem:partition_neighbors_cliques} the neighbors of any $g \in G$ can be partitioned into a set of disjoint cliques.
We denote by $m$ the number of these disjoint cliques and by $\delta_1 \le \delta_2 \le \dotsc \le \delta_m$ their sizes.
Clearly, $\delta \coloneqq \sum_{i=1}^m \delta_i$ is the degree of $\Cay(G, \Sigma)$ and thus $\delta = \abs{\Sigma}$.

\begin{lemma}\label{lem:cliques_number_and_size}
    Let $\Cay(G, \Sigma)$ be a counterexample to \Cref{conj:main}. 
    Then $m \ge 1 + \delta_m \ge 1 + \delta / m$ and $m \ge 3$.
\end{lemma}

\begin{proof}
  The Cayley graph $\Cay(G, \Sigma)$ contains a clique of size $\delta_m + 1$.
  By~\cite[Corollary 1]{gorovoy2021graphs} the neighborhood of some and, hence, of every vertex contains an independent set of size $\delta_m + 1$.
  Thus $m \ge \delta_m + 1$ by definition of $m$.
  The second inequality follows from $\delta = \sum_{i=1}^m \delta_i \le \sum_{i=1}^m \delta_m = m \delta_m$.

  For the final inequality, note that $m \le 1$ holds if and only if $\Cay(G, \Sigma)$ is complete.
  If $m=2$, then $\Cay(G, \Sigma)$ is a cycle.
  Neither can be a counterexample to \Cref{conj:main}.
  As such, $m \geq 3$.
\end{proof}

We define the function $\alpha$, which will take an important role as parameter for the size of a two-ball in $\mathrm{Cay}(G, \Sigma)$ as follows: 
\begin{equation*}
    \alpha(m_0) \coloneqq \frac{3m_0 - 4}{2m_0 - 2}\text{.}
\end{equation*}

\begin{lemma}\label{lem:two_ball_size}
    Let $\mathrm{Cay}(G, \Sigma)$ be a counterexample to \Cref{conj:main} and let $\alpha = \alpha(m)$.
    Then, for every $g \in G$, the size of the two-ball centered at $g$ satisfies \[
    \abs{\cB_2(g)} 
    = 1 + \delta + \delta^2 - \sum_{i=1}^m \delta^2_i 
    \geq 1 + \delta + \tfrac12 \alpha \delta^2\text{.}
    \]
\end{lemma}
\begin{proof}
    Every vertex at distance one from $g$ which is contained in a clique of size $\delta_i$ has $\delta-\delta_i$ neighbors at distance two from $g$.
    As each vertex at distance two from $g$ has a unique neighbor at distance one from $g$, we obtain the equality
  \begin{equation*}
    \abs{\cB_2(g)} = 1 + \delta + \sum_{i=1}^{m} \delta_i(\delta - \delta_i ) = 1 + \delta + \delta^2 - \sum_{i=1}^{m} \delta_i^2.
  \end{equation*}

  We then apply the reversed Cauchy-Schwarz inequality due to 
  P{\'o}lya and Szeg{\"o}.
  \cite[pp.\ 57, 213--214]{PolyaSzego1925} with the bounds $1 \leq \delta_i \leq m - 1$ obtained in \Cref{lem:cliques_number_and_size}.
  This yields
  \[
    m\sum_{i=1}^m \delta_i^2 
    = \left(\sum_{i=1}^m 1^2\right)\cdot\left(\sum_{i=1}^m \delta_i^2\right) 
    \le \frac14 \frac{m^2}{m-1} \left(\sum_{i=1}^m \delta_i\right)^2 
    = \frac14\frac{m^2}{m-1}\delta^2,
  \]
  from which the claimed inequality follows.
\end{proof}

\begin{lemma}\label{lem:two_ball_intersection}
    Let $g,h \in G$. 
    If $d(g,h) \ge 3$, then $\abs{\cB_2(g) \cap \cB_2(h)} \le 2 \delta - 1$.
\end{lemma}
\begin{proof}
    Let $h'$ be the vertex preceding $h$ on the geodesic path from $g$ to $h$.
    Then $\cB_2(g) \cap \cB_1(h) \sse \{ h' \}$ by uniqueness of geodesics.
    Hence, $\abs{\cB_2(g) \cap \cB_1(h')} \le \abs{\cB_1(h') \setminus \{h\}} = \delta$.
    If $h'' \in \cN(h) \setminus \{h'\}$, then $d(g,h'') \ge 3$, which implies $\abs{\cB_2(g) \cap \cB_1(h'')} \le 1$.
    In total, we obtain $\abs{\cB_2(g) \cap \cB_2(h)} \le 2\delta - 1$.
\end{proof}

We now have all the tools necessary to prove the main result of this section.

\begin{theorem}\label{thm:bound_generating_set}
    Let $\mathrm{Cay}(G, \Sigma)$ be a counterexample to \Cref{conj:main} of diameter at least three.
    Then $\alpha_0\abs{\Sigma}^2 < \abs{G}$ holds for all $\alpha_0 = \alpha(m_0)$ such that $m_0\ge 3$ is any integer with $\frac12 (m_0 - 1) (3m_0 - 4)(m_0 - 2)^2 < \abs{G}$.
\end{theorem}

Recall that the assumption on the diameter is satisfied whenever $G$ has nontrivial center (\cref{lem:center_no_diam_two}), even order (\cref{prop:notCG}), or order at most $2025$ (\cref{lem:small_order_diameter_two_implies_c5}).
In case $\mathrm{Cay}(G, \Sigma)$ has diameter two, we obtain a similar but weaker bound as follows.
As in the proof of \cref{lem:small_order_diameter_two_implies_c5}, we have $\abs{\Sigma} \geq (\lambda + 1)(\lambda + 13)$ for some $\lambda \geq 2$ as well as $\abs{\Sigma} \geq 46$.
The former yields the inequality 
\[
\lambda \leq \frac{\sqrt{144 + 4\abs{\Sigma}} - 14}{2} \leq \frac{\sqrt{144} + \sqrt{4\abs{\Sigma}} - 14}{2} 
= \sqrt{\abs{\Sigma}} - 1.
\]
Then $\abs{G} = \abs{\Sigma}^2 - \lambda \abs{\Sigma} + 1 > \abs{\Sigma}^2 - \abs{\Sigma}\sqrt{\abs{\Sigma}} + \abs{\Sigma}$ by Equation~\eqref{eq:diam_to_vertex_count} and using the above.
Since $\abs{\Sigma} \geq 46$, a direct computation shows that $\beta_0\abs{\Sigma}^2 < \abs{G}$ where $\beta_0 = \frac{47}{46} - \frac{1}{\sqrt{46}} \approx 0.874$ giving the bound in \cref{thm:bound_generating_set+simplified}.

\begin{remark}
    Before proceeding with the proof, we note that the factor $\alpha_0 = \alpha(m_0)$ is monotonically increasing in $m_0$.
    The choice $m_0 = 3$ is always valid for a counterexample to \Cref{conj:main} and yields $\alpha_0 = \tfrac54$.
    On the other hand, $\alpha_0 \to \frac32$ as $m_0 \to \infty$.
    Moreover, as a byproduct of our computer experiments, we obtain \emph{a posteriori} that the conclusion of \Cref{thm:bound_generating_set} holds with $\alpha_0 = \frac75$:
    after verifying that \Cref{conj:main} holds for every group of order at most $560$ (see \Cref{thm:mainCompSearch}), we set $m_0 = 6$.

    We use a combination of the bound obtained in \Cref{thm:bound_generating_set}, as well as other ones discussed hereafter, in our computer experiments.
    This results in a massive reduction in the number of generating sets we have to consider.
    In order to make these bounds as tight as possible, we choose the maximal value $m_0$ that is permitted for the group currently under examination.
    For some groups, we also employ \Cref{thm:center_geodesics} to establish an improved lower bound on the diameter which, in turn, allows us to replace \Cref{lem:two_ball_intersection} with a better estimate.
\end{remark}

\begin{proof}
    For the sake of deriving a contradiction, we assume that $\alpha_0 \abs{\Sigma}^2 \ge \abs{G}$.
    We also continue to employ the notation established above.
    In particular, the previous inequality becomes $\alpha_0\delta^2 \ge \abs{G}$, and \Cref{lem:cliques_number_and_size} then yields $\alpha_0 (m^2 - m)^2 \ge \alpha_0 \delta^2 \ge \abs{G}$.
    We cannot have $m < m_0$, for otherwise
    \[
        \tfrac12 (m_0 - 1)(3m_0 - 4)(m_0 - 2)^2 = \alpha_0 \big((m_0 -1)^2 - (m_0 - 1)\big)^2 \ge \alpha_0(m^2 - m)^2 \ge \abs{G}
    \]
    which contradicts the choice of $m_0$.
    Hence $m \geq m_0$ and, therefore, also $\alpha = \alpha(m) \ge \alpha(m_0) = \alpha_0$. 

    Finally, recall that $\mathrm{Cay}(G, \Sigma)$ has diameter at least three by assumption.
    Hence there exist elements $g,h \in G$ with $d(g, h) = 3$.
    Using \Cref{lem:two_ball_size} and \Cref{lem:two_ball_intersection} we arrive at 
    \[
        \abs{\cB_2(g) \cup \cB_2(h)} 
        = \underbrace{\abs{\cB_2(g)} + \abs{\cB_2(h)}}_{\ge 2 + 2\delta + \alpha \delta^2} - \underbrace{\abs{\cB_2(g) \cap \cB_2(h)}}_{\le 2\delta - 1} \ge 3 + \alpha \delta^2 
        \ge 3 + \alpha_0 \delta^2 \ge 3 + \abs{G}.
    \]
    As such, $\abs{G} \ge \abs{\cB_2(g) \cup \cB_2(h)} \ge \abs{G} + 3 > \abs{G}$, which is the desired contradiction.
\end{proof}

\begin{corollary}\label{cor:bound_center_naive}
    Let $\mathrm{Cay}(G, \Sigma)$ be a counterexample to \Cref{conj:main}.
    Then $\alpha_0\big(\lvert{Z(G)}\rvert - 1\big)^2 < \abs{G}$.
\end{corollary}

\begin{proof}
    This follows from \Cref{thm:bound_generating_set}, as $\abs{\Sigma} \ge \abs{Z(G)} - 1$ by \Cref{thm:center_geodesics} and \Cref{cor:center}.
Note that we can assume that $Z(G) \neq 1$; \cref{lem:center_no_diam_two} then excludes the case that $\mathrm{Cay}(G, \Sigma)$ has diameter two.
\end{proof}

In \Cref{thm:bound_generating_set} we have used two balls of radius two centered at vertices at distance three from each other to give a lower bound to the size of the group.
In cases where the group has a non-trivial center, we can improve upon this.
By \Cref{thm:center_geodesics}, the distance between any two central elements is at least three.
Using a subset of these as centers of balls of radius two, and \Cref{lem:two_ball_intersection} to bound the size of pairwise intersections, we arrive at the following.

\begin{theorem}\label{thm:bound_generating_set+center}
    Let $\mathrm{Cay}(G, \Sigma)$ be a counterexample to \Cref{conj:main}.
    Then the inequality
    \begin{equation}\label{eq:bound_generating_set+center}
        \tfrac12 \alpha_0 z_0 \abs{\Sigma}^2 - z_0(z_0 - 2)\abs{\Sigma} + \tfrac12 z_0(z_0 + 1) \le \abs{G},
    \end{equation}
    where $\alpha_0 = \alpha(m_0)$ holds for all integers $z_0$ with $1 \le z_0 \le \abs{Z(G)}$ and all lower bounds $m_0$ on the number $m \geq m_0$ of maximal cliques in the neighborhood $\cN(1)$ of $1\in G$.
\end{theorem}

\begin{proof}
    Choose a set of $z_0$ central elements, say $Z_0 \sse Z(G)$.
    By \Cref{lem:two_ball_size} and \Cref{lem:two_ball_intersection},
    \[
        \abs{G} \ge
        \Big\lvert\;\bigcup_{\mathclap{g \in Z_0}} \cB_2(g)\;\Big\rvert
        \ge 
        \sum_{\mathclap{g \in Z_0}} \abs{\cB_2(g)} - \sum_{\mathclap{\substack{g,h \in Z_0\\g \neq h}}} \abs{\cB_2(g) \cap \cB_2(h)} 
        \ge z_0\big(1 + \delta + \tfrac12\alpha\delta^2\big) - \binom{z_0}{2}\big(2\delta - 1\big)
    \]
    with $\alpha \ge \alpha_0$. (Recall that $\delta = \abs\Sigma$.)
    Expanding the rightmost expression yields the stated inequality.
\end{proof}

Note that in \Cref{thm:bound_generating_set+center} we have a family of inequalities, parameterized by the number $z_0$ of central elements used for placing the two-balls.
For groups with a small center, the best lower bound is achieved by $z_0 = |Z(G)|$.
However, if the center is large enough, then at some point it is no longer beneficial to place more balls, due to the way we approximate the intersection.
The best lower bound is obtained for $z_0 \approx \frac{1}{4}\alpha_0\delta$.
Based on this observation we obtain the following bound on the center.

\begin{corollary}\label{cor:bound_center}
  Let $\mathrm{Cay}(G, \Sigma)$ be a counterexample to \Cref{conj:main}.
  Then $\frac{1}{16}\alpha_0^2\big(\lvert{Z(G)}\rvert - 1\big)^3 \le \abs{G}$.
\end{corollary}

\begin{proof}
If $\lvert{Z(G)}\rvert = 1$, then the statement holds.
If $\lvert{Z(G)}\rvert > 1$, the inequality is derived from \Cref{thm:bound_generating_set+center} by setting $z_0 = \big\lceil \frac{1}{4} \alpha_0 \big(\abs{Z(G)} - 1\big)\big\rceil$:
    \begin{align*}
      \abs{G} &\ge (\tfrac12 \alpha_0\delta - z_0) z_0 \delta + 2z_0 \delta\\
      &\ge (\tfrac12 \alpha_0 \delta - \tfrac14 \alpha_0 (Z(G) - 1) - 1) z_0 \delta + 2 z_0 \delta\\
      &\ge (\tfrac12 \alpha_0 \delta - \tfrac14 \alpha_0 (Z(G) - 1)) z_0 \delta\\
      &\ge \tfrac{1}{16} \alpha_0^2 (Z(G) - 1)^3.
    \end{align*}
    The first inequality follows directly from \Cref{eq:bound_generating_set+center}
    by dropping the terms not containing $\delta = \abs{\Sigma}$, which are all positive.
    The second inequality uses $z_0 = \big\lceil \frac{1}{4} \alpha_0 \big(\abs{Z(G)} - 1\big)\big\rceil \le \frac{1}{4} \alpha_0 \big(\abs{Z(G)} - 1\big) + 1$.
    The third inequality is obtained by dropping the term $2z_0\delta - z_0 \delta \geq 0$.
    Finally we use the inequalities $z_0 \ge \frac{1}{4} \alpha_0 \big(\abs{Z(G)} - 1\big)$ and $\delta \ge \abs{Z(G)} - 1$ to obtain the desired statement.
\end{proof}

To obtain the result above, we placed two-balls on central elements.
If the central elements are sufficiently far apart, then we can show an even stronger bound, by placing the two-balls along the central geodesics (geodesics connecting the central elements).
Recall that by \Cref{thm:center_geodesics} central geodesics all have the same length $k$.
We begin with the following lower bound on the distance between two vertices on two different central geodesics.

\begin{lemma}\label{lem:vertex-on-central-geodesic-distance}
    Let $\mathrm{Cay}(G, \Sigma)$ be a counterexample to \Cref{conj:main} with $Z(G) \neq \Triv$.
    Let $g$ be a vertex on a central geodesic with endpoints $y_1$ and $y_2$.
    Let $h$ be a vertex on a different central geodesic with endpoints $z_1$ and $z_2$.
    Then $d(g,h) \ge \max \{ i, j \}$, where $i= \min \{ d(y_1, g), d(y_2, g) \}$ and $j = \min \{ d(z_1, h), d(z_2, g)\}$.
\end{lemma}

\begin{proof}
    Let $k$ be the length of the central geodesics.
    As $g$ and $h$ lie on different central geodesics, at most one pair of the vertices $\{y_1, y_2, z_1, z_2\}$ may coincide.
    We choose $y \in \{y_1, y_2\}$, $z \in \{z_1, z_2\}$ and $a, b \in \Sigma$ such that $\{y_1, y_2\} = \{y, ya^k\}$, $\{z_1, z_2\} = \{z, zb^k\}$, and if any of the four endpoints coincide, then $ya^k = zb^k$.
    A consequence of this choice is $y \neq z$, $y \neq zb^k$, $d(y, g) \le k - i$, and $d(z, h) \le k - j$.

    We claim that $d(y, h) \ge k$.
    Let $w$ be the geodesic of $y^{-1}h$.
    Observe that $y^{-1}h$ commutes with $b$, as it is the product of the central element $y^{-1}z$ with a power of $b$.
    In the case that $w \neq b^{\pm m}$ where $m = |w|$, by \Cref{lem:commgen_coset}, the geodesic of $y^{-1}z \in y^{-1}h\gen{b}$ has length at most $m$, that is $k = d(1, y^{-1}z) \le m = d(1, y^{-1}h)$.
    We obtain the desired statement $d(y, h) \ge k$ using vertex transitivity.
    In the case that $w = b^{\pm m}$, assume for a contradiction that $m < k$.
    Then $y^{-1}h = b^{\pm m}$ is on the central geodesic between $1$ and $b^{\pm k}$.
    By vertex transitivity, $h$ is on the 
    central geodesic between $y$ and $yb^{\pm k}$.
    Thus $y \in \{z, zb^k\}$ contradicting
    the choice~of~$y$.\par
    
    From the triangle inequality $d(y, h) \le d(y, g) + d(g, h)$ we obtain $d(g, h) \ge d(y, h) - d(y, g) \ge k - (k-i) = i$. 
    Similarly, using $d(z, g) \le d(z, h) + d(h, g)$, we obtain $d(g, h) \ge j$.
\end{proof}

Multiple disjoint balls of radius two can be placed along each central geodesic if its length permits.
However, we will only give an explicit bound for the simplest case, placing a single such ball in the middle of each central geodesic.

\begin{proposition}\label{prop:bound_center4}
    Let $\mathrm{Cay}(G, \Sigma)$ be a counterexample to \Cref{conj:main}. 
    If $3 \nmid |G|$ and $5,7 \nmid |G : Z(G)|$, then $\frac{1}{4}\alpha_0 \big(\abs{Z(G)} - 1\big)^4 < \abs{G}$.
\end{proposition}

\begin{proof}
    If $\abs{Z(G)} = 1$, then clearly the statement is true.
    By \Cref{cor:evencenter} the order of $Z(G)$ is odd.
    With our assumption $3 \nmid \abs{G}$, there must be an element of order at least $5$ in the center.
    Thus, the length $k$ of central geodesics must be odd by \Cref{lem:central_length_not_even}.
    As $5,7 \nmid \abs{G : Z(G)}$ we have $k \ge 11$.
    
    There are $\tfrac12 \abs{Z(G)} \big(\abs{Z(G)} -1\big)$ distinct pairs of central elements, and thus, that many central geodesics.
    We place a two-ball on the middle of each central geodesic.
    The intersection between any pair of two-balls is trivial, as the distance between any two of their center points is at least $5$ by \Cref{lem:vertex-on-central-geodesic-distance}.
    Thus we obtain
    \begin{align*}
        \abs{G} &\ge \tfrac12 \abs{Z(G)} \big(\abs{Z(G)} -1\big) \abs{\cB_2(1)}\\ 
        & >\tfrac12 \big(\abs{Z(G)} -1\big)^2 \big(\tfrac{1}{2} \alpha\delta^2 + \delta + 1\big) \\
        &> \tfrac14\alpha\big(\abs{Z(G)} - 1\big)^2\delta^2\text{.}
    \end{align*}
    Using the inequality $\delta \ge \abs{Z(G)} - 1$ yields the desired bound.
\end{proof}

\section{Further cases: dihedral, nilpotent, groups with large commutativity degree}\label{sec:furtherCases}

In this section we show that further large families of groups satisfy \Cref{conj:main} by combining results from \Cref{sec:theoryProofs} with more detailed knowledge about finite group theory and insights gleaned from the computer search.

We first consider groups that have an abelian subgroup of index two (which is necessarily normal); these include the dihedral groups. Pushing this to index three presents more challenges; so for that case we are able to prove the conjecture only in the two important special cases when the subgroup is not normal or when the center is trivial.
For nilpotent groups we can prove the conjecture holds in all groups of class two except for a particular subfamily (see \Cref{cor:SpecialNilpotent2}), and in groups of any class provided certain numerical conditions are satisfied (see \Cref{thm:nilpotent}).
Each of these families of groups in some sense is close to abelian, which is emphasized by the fact that these classes cover all groups with a high \emph{commutativity degree} (see \Cref{thm:commutativity_degree}).

\subsection{Abelian subgroups of index two}

\begin{lemma}\label{lem:index_two}
  Let $G$ be a group and $1 < N < G$ with $\lvert G : N \rvert = 2$.
If $\Cay(G, \Sigma)$ is geodetic, then $N \cap \Sigma \neq \emptyset$.
\end{lemma}
\begin{proof}
	If $\Sigma \cap N = \emptyset$, then every word $w \in \Sigma^*$ representing $1 \in G$ has even length. 
    Hence, all cycles in $\Cay(G, \Sigma)$ must have even length.
    But then $\Cay(G, \Sigma)$ cannot be geodetic (see \Cref{lem:even_cycle}).
\end{proof}

\begin{theorem}\label{thm:abelian_by_c2_split}
  Let $\phi \colon A \to A$ be an order-two automorphism of an abelian group $A = \langle X\mid R\rangle$. Let \[D_{A,\phi} \coloneqq A \rtimes_\phi C_2 =\langle X\cup\{t\}\mid R\cup\{t^2,txt(\phi(x))^{-1};\, x\in X\}\rangle\] be the corresponding semidirect product. 
  Then the only geodetic Cayley graph of the generalized dihedral group $D_{A,\phi}$ is the complete graph.
\end{theorem}

\begin{proof}
  Assume that $\Cay(D_{A,\phi}, \Sigma)$ is geodetic but not complete (in particular, $Z(D_{A,\phi})\cap \Sigma=\emptyset$ by \Cref{cor:center}).
  Since $\mathrm{ord}(t) = 2$, the generating set $\Sigma$ contains a conjugate of $t$ by \Cref{lem:conjugate_ord_two}.
  Upon replacing $\Sigma$ with a suitable conjugate if necessary, we may therefore assume that $t \in \Sigma$.

  As $\lvert D_{A,\phi} : A \rvert = 2$, there exists some $x \in A \cap \Sigma$ by \Cref{lem:index_two}.
  If $\phi(x) \in \{x^{\pm1}\}$, then $x^{D_{A, \phi}} \subseteq \{x^{\pm1}\} \subseteq \Sigma$.
  Since $\Sigma\neq \{x^{\pm1}\}$, \Cref{lem:conjugacyclass} implies that $\Cay(D_{A,\phi}, \Sigma)$ is complete.
  Thus $\phi(x)\not\in\{x^{\pm1}\}$. 
  Note that $\phi(x)x$ commutes with $t$ as 
  $t\phi(x)x=ttxtx=xtxtt=x\phi(x)t=\phi(x)xt$ as $x,\phi(x)\in A$.
  Hence $\phi(x)x \in Z(D_{A,\phi}) \setminus \{1\}$. 

  By \Cref{thm:center_geodesics}, there exists $y \in \Sigma$ such that $\phi(x) x = y^k$ with $k \geq 3$ and, therefore, with $k = 3$ by uniqueness of geodesics (since $\phi(x)x=txtx$).
  Thus, in particular, the length of central geodesics is three.  
  Now, suppose that $z\in \Sigma$ with $z^3\in Z(D_{A,\phi})$. 
  Then $z \in A$, for otherwise $z^3 \not\in A$ (since $A$ has index $2$) which would contradict the assumption that $z^3 \in Z(D_{A,\phi}) \leq A$.
  In particular, this shows that $y \in \Sigma \cap A$.
  
  Next, we argue that we can find an element $z \in \Sigma \cap A$ with $z \notin \{y^{\pm1}\}$. 
  If so, then $y$ and $z$ commute; hence $z^{\gen{y}}=\{z\} \sse \Sigma$ and then we can apply \Cref{lem:conjugacyclass} to show that $\gen{y}$ is complete.
  However, this implies $y^3 \in \Sigma$ and thus $Z(D_{A,\phi})\cap \Sigma\neq \emptyset$; contradicting the assumption that $Z(G) \cap \Sigma = \emptyset$.

    Firstly, suppose that $y=x$, then $x^3 = y^3 = \phi(x)x$ and thus $\phi(x) = x^2$. We obtain the contradiction $y^3 = x^3=\ov x x^4 = \ov x \phi(\phi(x)) = \ov x x = 1$.
    Next, suppose that $y = x^{-1}$ holds.
    Then $x^{-3} = y^3 = \phi(x)x$ and thus $\phi(x) = x^{-4}$, or equivalently $\phi(y) = y^{-4}$. Then $\phi(\phi(y))=\phi(y^{-4})=y^{16}$ and since $\phi$ is order two, this means $y=y^{16}$ so $y^{15}=1$.
    This means the order of $y$ is either $3,5$, or $15$. It cannot be $3$ since $y^3$ is a non-trivial element of the center.
    If $y^5=1$, then $y = y^{-4} =\phi(y) \in Z(D_{A,\phi})$ and so $Z(D_{A,\phi}) \cap \Sigma \neq \emptyset$ and we are done.
    Thus $\ord(y)=15$.
    We then have $y^6 \in Z(D_{A,\phi}) \setminus \{1\}$ and $y^6 \notin \{y^{\pm3}\}$.
    Thus $y^6 = z^3$ for some $z \in \Sigma \cap A$ with $z \notin \{y^{\pm1}\}$.
    Lastly if $y \neq x^{\pm1}$ then $z \coloneqq x$ satisfies our requirements. 
    Thus we have found an element $z \in \Sigma \cap A$ with $z \not\in \{y^{\pm1}\}$ which by the above paragraph shows that the Cayley graph is complete, contradicting our assumption.
\end{proof}

\begin{corollary}\label{cor:dihedral}
    The only geodetic Cayley graph of a dihedral group is the complete graph.
\end{corollary}

We will now reduce the general case, of admitting an abelian subgroup of index two, to the situation discussed in \Cref{thm:abelian_by_c2_split}.
In other words, we will prove the following.

\begin{theorem}\label{thm:abelian_by_c2}
    Let $G$ be a group.
    If there exists an abelian subgroup $A \leq G$ such that $\lvert G : A \rvert = 2$, then the only geodetic Cayley graph of $G$ is the complete graph.
\end{theorem}

Our reduction will rely on a certain relationship between the structure of the group $G$ and the parity of the order of $A$ as well as that of $Z(G)$.
Part of this relationship is captured by the following observation (with $p = 2$). 
It will be used again (with $p=3$) in \Cref{subsec:abelian_index_3}.

\begin{lemma}\label{lem:abelian_by_prime_center}
    Let $G$ be a group and $A \trianglelefteq G$ a normal abelian subgroup of prime index $\lvert G : A \rvert = p$.
    If $g \in G \setminus A$, then $g^p \in Z(G) \cap A$.
    Moreover, if $p \nmid \abs{A}$, then there exists some $\tilde g \in G \setminus A$ with $\tilde g^p = 1$.
\end{lemma}
\begin{proof}
    The image of any given $g \in G \setminus A$ in the quotient $G / A \cong C_p$ has order $p$ and, as such, $g^p \in A$.
    In particular, $g^p$ commutes with every element $a \in A$ and, clearly, $g^p$ also commutes with $g$.
    We have $\langle A, g \rangle = G$ since $A \le G$ is a maximal subgroup (its index is prime) and $g \in G \setminus A$.
    Therefore, the element $g^p$ commutes with all elements of $G$, \ie $g^p \in Z(G)$ as claimed.
    Lastly, since $p$ divides $\abs{G} = \abs{G : A}\abs{A}$, there exists an element $\tilde g \in G$ of order $p$ by Cauchy's theorem.
    If $p \nmid \abs{A}$, then $\tilde g \not\in A$ by Lagrange's theorem.
\end{proof}

\begin{proof}[Proof of \Cref{thm:abelian_by_c2}]
    If $\abs{Z(G)}$ is even, then \Cref{cor:evencenter} applies, so assume $\abs{Z(G)}$ is odd.
    We claim that $\abs{A}$ is odd as well.
    Suppose otherwise.
    Then there exists some $a \in A$ with $\mathrm{ord}(a) = 2$.
    Note that we cannot have $a \in Z(G)$ for $\abs{Z(G)}$ is odd.
    Let $g \in G \setminus A$ and consider $\tilde a \coloneqq a^g a \in A$.
    Clearly, $a^g \neq a$ for otherwise we would have $a \in Z(G)$.
    Therefore, $\tilde a \neq 1$.
    Since $\tilde a^2 = 1$, we conclude that $\ord(\tilde a) = 2$.
    Moreover $\tilde a^g = a^{gg} a^{g} = a a^g = a^g a = \tilde a$ since $g^2 \in Z(G)$ by \Cref{lem:abelian_by_prime_center}.
    But then we conclude that $\tilde a \in Z(G)$ and, therefore, $\abs{Z(G)}$ would have to be even.

    Finally, since $\abs{A}$ is odd, there exists $t \in G \setminus A$ with $t^2 = 1$ by \Cref{lem:abelian_by_prime_center}.
    Consider the automorphism $\phi \colon A \to A$ with $\phi(a) = a^t$.
    It satisfies $\phi^2 = \mathrm{id}$ and $\phi \neq \mathrm{id}$ since $t^2 = 1$ and $t \not\in Z(G)$, respectively.
    As such, it has order two and we can apply \Cref{thm:abelian_by_c2_split}.
\end{proof}

\subsection{Abelian subgroups of index three}\label{subsec:abelian_index_3}

For an index 3 abelian subgroup $A \leq G$ we are able to show that \cref{conj:main} holds for two cases: first if $A$ is not normal and second if the center of $G$ is trivial.

The general intuition is that, if a group has a lot of commuting elements (as it does when it contains an abelian subgroup of small index), then all of its Cayley graphs will need to have a lot of squares.
However, we can only manage to make this precise with particular hypotheses.
In \Cref{subsec:LargeComm} we investigate this intuition further.

\begin{lemma}\label{lem:abelian_index3_normal}
    Let $G$ be a group and suppose that there exists an abelian subgroup $A\leq G$ such that $\abs{G : A} = 3$ and $A$ is not normal in $G$.
    Then the only geodetic Cayley graph of $G$ is the complete graph.
\end{lemma}
\begin{proof}
    Let $A^g \le G$ be a conjugate subgroup with $A \neq A^g$.
    Then $\langle A, A^g \rangle = G$ since $A$ is a maximal subgroup of $G$ (as its index is prime).
    Moreover, we observe that $A \cap A^g \sse Z(G)$ as every element of $A \cap A^g$ commutes with every element of $\langle A, A^g \rangle$.
    
    Now consider the action of $G$ by left-multiplication on the set of cosets $G/A$.
    It gives rise to a homomorphism $\rho \colon G \to S_3$ with $\mathrm{Ker}(\rho) = \{h \in G \mid hgA = gA \text{ for all } g \in G\} = \bigcap_{g\in G} A^g \le Z(G)$.
    The $\rho$-preimage of $C_3 \le S_3$ is an abelian subgroup $\rho^{-1}(C_3) \le G$, as it is an extension of $\mathrm{Ker}(\rho) \leq Z(G)$ by a cyclic group.

    Since $\lvert G : \rho^{-1}(C_3) \rvert = 2$ as $\rho$ is surjective, the statement now follows from \Cref{thm:abelian_by_c2}.
\end{proof}

\begin{remark}\label{rmk:AbelIndex3Center}
   In the situation described in \Cref{lem:abelian_index3_normal}, we have $\lvert G : Z(G) \rvert=6$.
   In \Cref{cor:bound_center_naive} we have proved an inequality relating the size of the center and the size of the group, which in the case at hand gives us the bound $\lvert G \rvert < 36$ (assuming that $G$ violates \cref{conj:main}, we have $\tfrac{5}{4}(\abs{Z(G)} - 1)^2 < \abs{G} = 6 \abs{Z(G)}$ which, since $\abs{Z(G)}$ is necessarily odd by \cref{cor:evencenter}, implies that $\abs{Z(G)} < 6$  and thus $\abs{G} < 36$).
   These groups can be checked by a computer search.

   We note that this alternative proof can be adapted, so as to cover groups with an abelian but \emph{not} normal subgroup of index $5$ or $7$ (the center of such a group has index at most $20$ or $42$ and, thus, by \Cref{cor:bound_center} we need only consider such groups with order at most $300$ and $840$, respectively).
\end{remark}

The other case is that $G$ contains a normal abelian subgroup of index three. For this we can make progress when we restrict to the case where the center is trivial.

\begin{theorem}\label{thm:abelian_by_c3_centerless}
    Let $G$ be a group with $Z(G) = \Triv$ and suppose that there exists an abelian subgroup $A \le G$ such that $\lvert G : A \rvert = 3$.
    Then the only geodetic Cayley graph of $G$ is the complete graph.
\end{theorem}

\begin{proof}
    Assume that $\mathrm{Cay}(G, \Sigma)$ is geodetic but not complete.
    By \Cref{lem:abelian_index3_normal}, $A$ is a normal subgroup of $G$.
    Note that we then have $\mathrm{ord}(\psi) = 3$ for each $\psi \in G \setminus A$ by \Cref{lem:abelian_by_prime_center} and our assumption on $Z(G)$.
    In order to derive a contradiction, we establish two claims regarding the conjugacy classes of elements of $A$ with specific geodesics, in particular, elements of the set $C \coloneqq \{ g \in A \mid \mathrm{geod}(g) \equiv \psi_1\psi_2 \text{ with } \psi_1,\psi_2 \in \Sigma \setminus A \}$.

    \medskip

    \noindent\textit{Claim (1). 
    If $x \in \Sigma \cap A$, then $x^G = \{x, \psi_1\psi_2, \psi_2\psi_1\}$ for some $\psi_1,\psi_2 \in \Sigma \setminus A$ with $\psi_1\psi_2, \psi_2\psi_1 \in C$.}

    \medskip
    
    Since $A \le G$ is a maximal subgroup, there exists some $\psi \in \Sigma \setminus A$.
    Because $G = A \cdot \{1,\psi,\psi^{-1}\}$, we know that $x^G = \{ x, \psi^{-1}x\psi, \psi x\psi^{-1} \}$.
    We also note that $\psi x \psi x \psi x = (\psi x)^3 = 1$ as $\psi x \in G \setminus A$.
    Therefore, the element $\psi x \psi = x^{-1}\psi^{-1}x^{-1}$ has length at most two.
    In turn, we conclude that both of the elements $\psi^{-1} x \psi = \psi (\psi x \psi)$ and $\psi x \psi^{-1} = (\psi x \psi) \psi$ have length at most two.
    
    If one of these elements, $\psi^{-1}x\psi$ say, had length one, then, so would $x\psi$ and $\psi^{-1}x$ by \Cref{lem:diamonds}; hence, so would $\psi(x\psi) = x^{-1}(\psi^{-1}x^{-1})$ and $(\psi^{-1} x)\psi^{-1} = (x^{-1}\psi)x^{-1}$; hence, so would $\psi x$ and $x \psi^{-1}$.
    It follows that the other element $(\psi x) \psi^{-1} = \psi(x \psi^{-1})$ would also have length one.
    But then all elements of the conjugacy class $x^G$ would have length one, \ie $x^G \sse \Sigma$.
    This contradicts \Cref{lem:conjugacyclass}.
    
    Next, we assume that $\mathrm{geod}(\psi^{-1} x \psi) \equiv y_1 y_2$ with $y_1,y_2 \in \Sigma \cap A$.
    This implies $y_1,y_2 \in \{ x^{\pm 1} \}$, for otherwise $\langle \Sigma \cap A \rangle \le G$ is a complete subgroup by \Cref{lem:conjugacyclass} and thus $y_1y_2$ has length one.
    If $\psi^{-1}x\psi = x^{\pm 2}$, then either $\psi x \psi^{-1} = \psi_1\psi_2 \in C$ or, for the same reason as above, $\psi x \psi^{-1} = x^{\mp 2}$.
    
    In the second case, $x$ is conjugate to $x^{-1}$ by transitivity of conjugacy.
    Since $x$ is the only element of length one in $x^G$, this implies $x = x^{-1}$. 
    But then $\psi^{-1} x \psi = x^{\pm 2} = 1$, which is absurd.
    
    In the first case, \ie $x^G = \{x, \tilde x, \psi_1\psi_2 \}$ where $\tilde x = x^{\pm2}$, we have $\psi^{-1}x\psi = \tilde x = x^{\pm 2}$ and, hence, $x = x^{\pm 8}$ as $\psi^3 = 1$.
    If $\psi^{-1}x\psi = x^{-2}$, then $x^9 = 1$ and thus $\psi^{-1}x^3\psi = x^{-6} = x^3 \in Z(G)$.
    But this implies $x^3 = 1$ and, therefore, $x^{-2} = x \in \Sigma$; a contradiction.
    In the case $\psi^{-1}x\psi = x^2$, we first observe that $\psi_2\psi_1 = x$.
    Indeed, $\psi_2 \psi_1$ is conjugate to $x$ and cannot equal $x^2$ or $\psi_1\psi_2$ by uniqueness of geodesics.
    Using this, we obtain $\psi_2x\psi_2^{-1} = x^2$ since $\psi_2 (\psi_1\psi_2) \psi_2^{-1} = x$ and $\psi_2$ has order three.
    Now $\psi_1 x = \psi_2^{-1} \psi_2 \psi_1 x = \psi_2^{-1} x^2 = \psi_2^{-1} \psi_2x\psi_2^{-1} = x \psi_2^{-1} \in \Sigma$, since it has two expressions of length two.
    But then the same is true for $x^2 = \psi_2 (\psi_1 x)$; hence $x^2 \in \Sigma$, which contradicts our assumption that $x^2$ has length two.
    
    The only remaining possibility is that $\mathrm{geod}(\psi^{-1} x \psi) \equiv \psi_1 \psi_2$ and, hence, $\mathrm{geod}(\psi x \psi^{-1}) \equiv \psi_2 \psi_1$ for some $\psi_1, \psi_2 \in \Sigma \setminus A$. 
    This establishes Claim (1).

\medskip
    \noindent\textit{Claim (2).
    If $\psi_1, \psi_2 \in \Sigma \setminus A$ with $g = \psi_1\psi_2 \in C$, then for some $x \in \Sigma \cap A$ \[
        g^G = \{x, \psi_1\psi_2, \psi_2\psi_1\}
        \quad\text{or}\quad
        g^G = \{x^2, \psi_1\psi_2, \psi_2\psi_1\}.
    \]}
    
    Clearly, $\psi_1\psi_2$ and $\psi_2\psi_1$ are distinct elements of $g^G$. 
    Moreover, $\psi_2\psi_1$ also has length two by Claim~(1).
    The third element of $g^G$ is $h \coloneqq \psi_1^{-1}\psi_2\psi_1^{-1} = \psi_2^{-1}\psi_1\psi_2^{-1}$, which therefore has length at most two.
    The claim is trivial if $h$ has length one.
    Suppose that $h$ has length two and $\mathrm{geod}(h) \equiv \psi'_1\psi'_2$ with $\psi'_1,\psi'_2 \in \Sigma \setminus A$.
    Clearly, $\psi'_2\psi'_1 \neq \psi'_1\psi'_2$ and $\psi'_2\psi'_1 \in g^G$.
    Moreover, $\psi'_2\psi'_1$ also has length two by what we have just shown.
    It follows that $\{\psi'_1, \psi'_2\} = \{\psi_1, \psi_2\}$, which is clearly impossible.
    
    If $\mathrm{geod}(h) \equiv y_1y_2$ with $y_1, y_2 \in \Sigma \cap A$, then $y_1 = y_2$ by uniqueness of geodesics and the fact that $y_1y_2 = y_2y_1$.
    Hence $h = x^2$ with $x = y_1 = y_2 \in \Sigma$.
    This completes our proof of the claim.

    \medskip

    We finally derive the desired contradiction.
    To this end, recall that there exists some $\psi \in \Sigma \setminus A$.
    Furthermore, we cannot have $\Sigma = \{\psi^{\pm 1}\}$ for $G$ is not cyclic.
    As such, we either have $\Sigma \cap A \neq \emptyset$ or there exists some $\psi' \in \Sigma \setminus A$ with $\psi' \not\in \{\psi^{\pm1}\}$.
    In the latter case, $\Sigma \cap A \neq \emptyset$ by Claim~(2).
    
    Now, choose some $x \in \Sigma \cap A$ and let $\psi_1,\psi_2 \in \Sigma \setminus A$ with $x^G = \{x, \psi_1\psi_2, \psi_2\psi_1 \}$ as in Claim~(1).
    We then have $x = \psi_1^{-1}\psi_2\psi_1^{-1} = \psi_2^{-1}\psi_1\psi_2^{-1}$.
    Let $\psi_3 \coloneqq x\psi_1 = \psi_1^{-1}\psi_2 \in \Sigma$ and observe $\psi_3 \not \in \{\psi_1^{\pm1}, \psi_2\}$.
    In fact, we also have $\psi_3 = \psi_1^{-1}\psi_2 \neq \psi_2^{-1}$ for otherwise $\psi_2\psi_1 = 1$.
    Now note that $\psi_3^{-1}\psi_2^{-1} = x$ and thus $\psi_2^{-1}\psi_3^{-1} \in x^G \setminus \{x\} = \{ \psi_1\psi_2, \psi_2\psi_1\}$, which contradicts uniqueness of geodesics.
\end{proof}

\subsection{Nilpotent groups}
In every nilpotent group, certain iterated commutators evaluate to central elements.
This fact, together with our results concerning central elements developed in \Cref{subsec:center}, imposes restrictions on the structure of a nilpotent group with an alleged non-complete geodetic Cayley graph.
For groups of nilpotency class two, we obtain further restrictions based on a more detailed analysis of the involved commutator maps; see \Cref{cor:SpecialNilpotent2} below.

Recall that a group $G$ is \emph{nilpotent} if $G = \Triv$ or $G / Z(G)$ is nilpotent.
If $G$ is nilpotent, then there exists a number $s$ such that $[g_1, \dotsc, g_{s+1}] \coloneqq [[g_1, \dotsc, g_s], g_{s+1}] = 1$ for all $g_1, \dotsc, g_{s+1} \in G$.
The smallest such number $s$ is the \emph{nilpotency class} of $G$.
A group is nilpotent if and only if it is a direct product of $p$-groups; see \cite[Theorem~5.2.4]{Robinson96book}.
In particular, every nilpotent group of even order has even-order center.

\begin{theorem}\label{thm:nilpotent}
    Let $G$ be a nilpotent but not cyclic group of nilpotency class $s$ and suppose that
    \begin{equation*}
        p \nmid \frac{\exp(G)}{\exp(Z(G))}
    \end{equation*}
    for each odd prime $p < 3 \cdot 2^{s-1} - 2$.
    Then the only geodetic Cayley graph of $G$ is the complete graph.
\end{theorem}

\begin{proof}
  All finite nilpotent groups are direct products of $p$-groups.
  Thus, if $2$ divides $\abs{G}$, then there must be an order $2$ element in the center and the statement follows from \Cref{cor:evencenter}.
  If $G$ is abelian, then the statement is a consequence of \Cref{cor:abelian}.
  We now assume that $s > 1$ and $2 \nmid \abs{G}$, and that there exists a geodetic Cayley graph $\mathrm{Cay}(G, \Sigma)$ which is not complete.

  In every finite group of nilpotency class $s$, by~\cite[Lemma 2.6]{ClementMZ17nilpotent}, there are generators $a_1, \dots, a_s \in \Sigma$ such that $z = [a_1, a_2, \dots, a_s] \neq 1$.
  Further, since $G$ is nilpotent of class $s$, we have $z \in Z(G)$; hence, by \cref{cor:center} and \Cref{thm:center_geodesics}, $z = a^k$ with $a \in \Sigma$ and $k \ge 3$ equal to the order of $a$ in the quotient group $G / Z(G)$.
  Since $a^k$ is a geodesic, its length is shorter than the length of the commutator, which is $3\cdot 2^{s-1} - 2$.
  
  Let $p$ be a prime divisor of $k = pr$. 
  Then $p$ also divides $\abs{G : Z(G)}$, $\abs{G}$, and $\abs{Z(G)}$.
  Furthermore, $p$ is odd and $p \leq k < 3 \cdot 2^{s-1} - 2$.
  Let $p^n$ be the largest power of $p$ dividing $\exp(Z(G))$ (note that $n \geq 1$). 
  Let $\tilde z \in Z(G)$ with $\mathrm{ord}(\tilde z) = p^n$.
  Once more, $\tilde z = \tilde a^k$ for some $\tilde a \in \Sigma$ by \Cref{thm:center_geodesics}.
  Now consider the element $b = \tilde a^r$ (where $k = pr$ as above).
  Clearly, $(b^{p})^{p^{n}} = \tilde z^{p^{n}} = 1$.
  As such, $\mathrm{ord}(b) = p^{n+1}$ and, thus, $p^{n+1}$ divides $\exp(G)$. 
  We conclude that $p$ divides $\exp(G) / \exp(Z(G))$, which contradicts our assumption on $G$.
\end{proof}

\begin{remark}\label{rem:nilpotent}
    The condition in \Cref{thm:nilpotent} is satisfied if no odd prime $p < 3\cdot 2^{s-1} - 2$ divides the order of $G$ or, more generally, if $p^2$ does not divide $\exp(G)$ for any such prime $p$.
    If, for example, $G$ is a nilpotent group of nilpotency class two with $9 \nmid \exp(G)$, then $G$ satisfies \Cref{conj:main}; if $G$ is of nilpotency class three and neither $3$, $5$, nor $7$ divide $\abs{G}$, then $G$ satisfies \Cref{conj:main}.
\end{remark}

We now turn to the case of nilpotency class two.
In the following we always assume that $G$ is nilpotent group of nilpotency class two and that $\mathrm{Cay}(G, \Sigma)$ is a counterexample to \Cref{conj:main}.
Recall that, by \Cref{cor:evencenter}, this implies that $2 \nmid \abs{G}$ and, by \Cref{thm:nilpotent} (see also \Cref{rem:nilpotent}), that $3 \mid \abs{Z(G)}$.

Throughout, we will consider the subset $\Delta \coloneqq \{ x \in \Sigma \mid x^3 \in Z(G) \setminus \{1\} \} \subseteq \Sigma$.
Since $[x,y] \in Z(G) \setminus \{1\}$ for some $x,y \in \Sigma$ (see the proof of \Cref{thm:nilpotent}), the length of central geodesics is three according to \Cref{thm:center_geodesics}.
As such, the map $\alpha \colon G \to G$ given by $g \mapsto g^3$ induces a bijection from $\Delta$ onto $Z(G) \setminus \{1\}$.
As there exists an element of order three in $Z(G)$, there exists an element of order nine in $\Delta$; in particular, $\Delta \neq \emptyset$.

\begin{lemma}\label{lem:nilpotent_class2_commutator1}
    Let $x,y \in \Sigma$ with $x \in \Delta$ or $y \in \Delta$.
    Then $[x,y] = 1$ if and only if $\{x^{\pm1}\} = \{y^{\pm1}\}$.
\end{lemma}
\begin{proof}
    If $[x,y] = 1$ and $\{x^{\pm1}\} \neq \{y^{\pm1}\}$, then $\langle x, y \rangle \le G$ is a complete subgroup by \Cref{cor:commgen_subgroup}.
    If, furthermore, $x \in \Delta$, then $x^3 \in Z(G) \cap \Sigma$.
    But then $\mathrm{Cay}(G, \Sigma)$ is complete by \Cref{cor:center}.
\end{proof}

\begin{lemma}\label{lem:nilpotent_class2_commutator2}
    Let $x \in \Sigma$ and $y_1,y_2 \in \Sigma \setminus \{x^{\pm1}\}$ with $y_1 \in \Delta$.
    Then $[x, y_1] = [x, y_2]$ implies $y_1 = y_2$.
\end{lemma}
\begin{proof}
    Let $z \coloneqq [x, y_1] = [x, y_2] \in Z(G)$.
    Then $z \neq 1$ by \Cref{lem:nilpotent_class2_commutator1} and, thus, $z = a^3$ with $a \in \Delta$.
    Assuming $y_1 \neq y_2$, the element $xz = x^{y_1} = x^{y_2}$ has two distinct representatives of length three; therefore, it has length at most two.
    But then $\bar x = a$ by uniqueness of geodesics, as $\bar x (xz) = z = a^3$.
    In particular, this shows that $x = \bar a \in \Delta$ and that $x^{y} = xz = x^{-2}$ for $y \in \{ y_1, y_2 \}$.

    Since $2 \nmid \abs{G}$, the order of $x$ cannot be even.
    Let $\mathrm{ord}(x) = 2k + 1$ and note that $k \ge 4$ as $x  \in \Delta$.
    Then $(x^k)^{y} = x^{-2k} = x$ for $y \in \{ y_1, y_2\}$.
    As such, $x^k = y_1x{y_1^{-1}} = y_2x{y_2^{-1}}$ has length at most two.
    Since $x^k \notin \{ 1, x^{\pm1}, x^{\pm2} \}$, we can apply \Cref{lem:commgen_coset} (with $w$ being the geodesic of $x^k$) to conclude that all elements of the coset $x^k \gen{x}$ have length at most two.
    In particular $x^{-3} = z \in Z(G)$ has length at most two; a contradiction!
\end{proof}

As usual, we denote the \emph{commutator (or derived) subgroup} of $G$ by $G'=\gen{[a,b];a,b\in G}$, and the \emph{Frattini} subgroup of $G$ by $\Phi(G)$, \ie $\Phi(G)$ is the intersection of all maximal subgroups of $G$.

Using the above, we will now show that a nilpotent group $G$ of class two that violates \Cref{conj:main} would have to be \emph{special}, \ie it would satisfy $G' = \Phi(G) = Z(G)$.
In some sense, such groups are as non-abelian as possible given the constraint $G' \leq Z(G)$ imposed by $G$ being nilpotent of class two.

\begin{proposition}\label{prop:nilpotent_class2}\label{cor:SpecialNilpotent2}
    Suppose that $G$ is a nilpotent group of class two and a counterexample to \Cref{conj:main}. 
    Then $\mathrm{exp}(G) = 9$ and $G' = \Phi(G) = Z(G) \not\cong C_3$, \ie $G$ is a special but not extra-special $3$-group.
\end{proposition}

\begin{proof}
    Suppose that $\mathrm{Cay}(G, \Sigma)$ violates \Cref{conj:main} and, as above, let $\Delta \coloneqq \{ x \in \Sigma \mid x^3 \in Z(G) \setminus \{1\} \} \subseteq \Sigma$ be the set of central roots.
    The subgroup $H \coloneqq \langle \Delta \rangle$ clearly satisfies $H' \le G' \le Z(G) \le H \le G$.
    Furthermore, $H / Z(G)$ is an elementary abelian $3$-group.
    We treat the cases $\Sigma \setminus \Delta = \emptyset$ and $\Sigma \setminus \Delta \neq \emptyset$ separately.
    
    We first deal with the case $\Sigma \setminus \Delta = \emptyset$, \ie $\Sigma = \Delta$ and $H = G$.
    As such, $G / Z(G)$ is an elementary abelian $3$-group and thus so is $G'$ (which is the image of $G/Z(G) \times G/Z(G)$ under the bilinear map induced by the commutator map).
    To show that $G' = Z(G)$, fix $x \in \Sigma$ and consider the homomorphism 
    \[ \psi \colon G \to Z(G);\quad g \mapsto [x, g]. \]
    By  \Cref{lem:nilpotent_class2_commutator2}, $\psi$ is injective on $\Sigma \setminus \{x^{\pm 1}\}$.
    But then $Z(G) \setminus \psi(G)$ contains at most two elements.
    From $\psi(G) \neq \Triv$, we conclude that $\psi(G) = Z(G)$ by Lagrange's Theorem.
    Hence, $G' = Z(G)$.
    From this we conclude that $\exp(G) = 9$ and that the Frattini subgroup $\Phi(G)$ coincides with $G' = Z(G)$ (in $p$-groups, the Frattini subgroup is the smallest subgroup with elementary abelian quotient; hence, $G' \leq \Phi(G) \leq Z(G)$ since $G / Z(G)$ is elementary abelian).
    Finally, note that $\abs{Z(G)} - 1 = \abs{\Delta} = \abs{\Sigma} \geq 4$ since $G$ is not cyclic.
    Therefore, we can conclude that $Z(G) \not\cong C_3$.
    This completes the proof for the case $\Sigma \setminus \Delta = \emptyset$.

    \smallbreak
    
    We now assume that there exists an element $u \in \Sigma \setminus \Delta$ and consider the homomorphism 
    \[ \phi \colon G \to Z(G);\quad g \mapsto [u, g]. \]
    By \Cref{lem:nilpotent_class2_commutator1}, $\phi(\Delta) \subseteq Z(G) \setminus \{1\}$ and, by \Cref{lem:nilpotent_class2_commutator2}, $\phi$ is injective on $\Delta$.
    Since $\abs{\Delta} = \abs{Z(G) \setminus \{1\}}$, we then have $\phi(\Delta) = Z(G) \setminus \{1\}$ and, therefore, $\phi(H) = \mathop{\mathrm{Im}}(\phi) = Z(G)$.
    We claim that $\phi(\Sigma \setminus \Delta) = \{1\}$.
    Indeed, if $y_2 \in \Sigma \setminus \Delta$ would satisfy $\phi(y_2) \neq 1$, then $\phi(y_2) = \phi(y_1)$ for some $y_1 \in \Delta$.
    Hence, $y_1 = y_2$ by \Cref{lem:nilpotent_class2_commutator2}.
    
    The homomorphism $\phi$ factors through the quotient $G \to G/Z(G)$ and the inclusion $G' \to Z(G)$.
    As such, $G' = Z(G)$ is isomorphic to a quotient of $H / Z(G)$; hence, $Z(G)$ is an elementary abelian $3$-group.
    It follows that $\Delta \sse \Sigma_2$, \ie $\mathrm{ord}(x) = 9$ for all $x \in \Delta$.
    On the other hand, we have $u^3 \in Z(G)$ since $[u^3, x] = [u, x]^3 = \phi(x)^3 = 1$ for each $x \in \Sigma$.
    This implies $u^3 = 1$ since $u \in \Sigma \setminus \Delta$.
    Moreover, the same holds for all $u \in \Sigma \setminus \Delta$.
    Since all elements of $\Sigma$ have order three in the quotient $G / Z(G)$, the latter is an elementary abelian $3$-group.
    As in the previous case, we conclude that $\mathrm{exp}(G) = 9$ and $G' = \Phi(G) = Z(G)$.

    It remains to show that $Z(G) \not\cong C_3$. 
    By way of contradiction, let us assume that $Z(G) \cong C_3$.
    Then $\Delta$ consists of precisely two elements, \ie $\Delta = \{ x^{\pm 1} \}$, and $\phi(\Delta) = Z(G) \setminus \{1\} = \{ x^{\pm3} \}$.
    Upon replacing $u$ by $u^{-1}$ if necessary, we may then assume that $[x,u] = x^{-3}$.
    It follows that $uxu = u^{2}x^{-2} = u^{-1}x^{-2}$ and this element has length at most two by uniqueness of geodesics.
    Hence, $(uxu)u = uxu^{-1}$ also has length at most two and so does $x^3 = x^{-1}(uxu^{-1})$.
    But this contradicts \Cref{thm:center_geodesics} since $x^3 \in Z(G) \setminus \{1\}$.
    As such, $Z(G) \not\cong C_3$.
\end{proof}

To further explore the applicability of \Cref{thm:nilpotent} and \Cref{cor:SpecialNilpotent2}, we have examined all non-abelian groups of order $p^k$ with $p \in \{3,5,7,11\}$ and $k \leq 7$ in GAP \cite{GAP4} using its SmallGrp library \cite{SmallGrp}.
The results of this examination are summarized in \Cref{tbl:nilpotent}.
Note that such groups have nilpotency class $s \leq 6$.
Moreover, by \Cref{thm:nilpotent}, every $p$-group of nilpotency class~$s$ with $p \geq 3 \cdot 2^{s-1} - 2$ satisfies \Cref{conj:main}; this threshold is indicated by horizontal lines in the table.
\begin{table}[h!]\centering
    \begin{tabular}{ccccccc}
    \toprule
    $\quad p \quad$ & $s\leq 1$ & $s=2$ & $s=3$ & $s=4$ & $s=5$ & $s=6$ \\
    \midrule
    $3$ & $45$ & $587 \,/\, 1{,}926$ & $150 \,/\, 6{,}362$ & $36 \,/\, 1{,}386$ & $0 \,/\, 180$ & $0 \,/\, 6$ \\
    \cmidrule{3-3}
    $5$ & $45$ & $7{,}256$ & $247 \,/\, 23{,}073$ & $119 \,/\, 3{,}382$ & $\phantom{0}5 \,/\, 1{,}227$ & $0 \,/\, 9$ \\
    $7$ & $45$ & $26{,}914$ & $255 \,/\, 76{,}783$ & $131 \,/\, 8{,}034$ & $60 \,/\, 2{,}140$ & $15 \,/\, 198$ \\
    \cmidrule{4-4}
    $11$ & $45$ & $204{,}912$ & $514{,}627$ & $139 \,/\, 26{,}882$ & $62 \,/\, 5{,}170$ & $19 \,/\, 402$ \\
    \bottomrule
    \end{tabular}
    \caption{The number of groups of order $p^k$ with $k \leq 7$ and nilpotency class $s$}{covered by \Cref{thm:nilpotent} and \Cref{cor:SpecialNilpotent2} / total number of such groups (if distinct).}
    \label{tbl:nilpotent}
\end{table}

\subsection{Groups with large commutativity degree}\label{subsec:LargeComm}

Many of our results address \Cref{conj:main} in the case of groups that are, in some sense, close to being abelian.
This property can be quantified by a group $G$'s \emph{commutativity degree} $\mathbf{P}(G)$, which is the probability  that two randomly chosen elements of $G$  commute, \ie
\[
    \mathbf{P}(G) \coloneqq \frac{\lvert \{ (g,h)\in G\times G \mid gh = hg \}\rvert}{\lvert G \times G \rvert}.
\]

The interested reader is referred to the survey by Das, Nath, and Pournaki \cite{DNP13CommutativityDegree} for further details and historical context.
We will content ourselves here with the following observation.

\begin{theorem}\label{thm:commutativity_degree}
    \Cref{conj:main} holds for every group $G$ with $\mathbf{P}(G) > \frac{11}{32}$.
\end{theorem}
\begin{proof}
    For abelian groups, this follows from \Cref{cor:abelian} (see also \cite{wrap108882}).
    If $G$ is non-abelian and $\mathbf{P}(G) > \frac{11}{32}$, then according to Rusin \cite[p.\ 246]{Rusin79} the structure of $G$ must be as indicated in \Cref{tbl:commutativity_degree_11_32}.

    \begin{table}[!ht]\centering
    \begin{tabular}{cccc}
        \toprule
        \multicolumn{3}{c}{\textbf{Group Structure}} & \textbf{Applicable Result} \\
        \cmidrule{1-3}
        $G'$ & $G' \cap Z(G)$ & $G / Z(G)$ & \\
        \midrule
        $C_2$ & $C_2$ & $C_2^{2r}\; (r \geq 1)$ & \Cref{cor:evencenter} \\
        $C_3$ & $\Triv$ & $S_3$ & \Cref{thm:abelian_by_c2} \\
        $C_2^2$ or $C_4$ & $C_2$ & $D_8$ & \Cref{cor:evencenter} \\
        $C_2^2$ & $C_2^2$ & $C_2^3$ or $C_2^4$ & \Cref{cor:evencenter} \\
        $C_3$ & $C_3$ & $C_3^2$ & \Cref{cor:SpecialNilpotent2} \\
        $C_5$ & $\Triv$ & $D_{10}$ & \Cref{thm:abelian_by_c2} \\
        $C_6$ & $C_2$ & $S_3 \times C_2$ or $C_3 \rtimes C_4$ & \Cref{cor:evencenter} \\
        \bottomrule
    \end{tabular}
    \caption{The possible structures of a non-abelian group $G$ with $\mathbf{P}(G) > \frac{11}{32}$.
    }\label{tbl:commutativity_degree_11_32}
    \end{table}

    In most of these cases, $Z(G)$ contains a subgroup isomorphic to $C_2$; hence, \Cref{cor:evencenter} applies.
    If $G/Z(G) \cong S_3$ or $G / Z(G) \cong D_{10}$, then $G$ contains an abelian subgroup of index two corresponding to $C_3 \leq S_3$ or $C_5 \leq D_{10}$ in $G / Z(G)$, respectively.
    As such, we can apply \Cref{thm:abelian_by_c2}.
    In the remaining case, \ie if $C_3 \cong G' \leq Z(G)$, then $G$ is covered by \Cref{cor:SpecialNilpotent2}. 
\end{proof}

\section{Experiments.}\label{sec:experiments}

We now turn to the exhaustive computer search to check which groups of order up to $1024$, as well as all even orders up to $2014$ and all non-abelian finite simple groups up to order $5000$, have a generating set that yields a geodetic Cayley graph.
The aim of this experiment was to either verify \Cref{conj:main} for as many group orders as possible, or to find a group and a generating set that yields a non-trivial geodetic Cayley graph, \ie a graph that is neither complete nor an odd cycle.
For our code, see

\centerline{\url{https://osf.io/9ay6s/?view_only=37e18301e4e74e12bfe4e07b90b924c0}.}

Improving the computer search has been a major motivation for the theoretical work in the previous sections.
Important results in this regard are the bounds on the generating set discussed in \Cref{sec:bounds} as well as the results covering entire classes of groups, the most important of which is \Cref{cor:evencenter} excluding all groups with even-order center.

In turn, the results from the computer search also influenced our theoretical work.
For example, at one point in the development of our computer search algorithms, groups which were semidirect products with $C_2$ had a long running time. 
This directed our theoretical work to focus on such groups, leading to \Cref{thm:abelian_by_c2} showing that \Cref{conj:main} holds for all groups with an abelian subgroup of index two.

\subsection{Overview}

Our approach consists of three stages.
The first is the \emph{filtering stage} in which we identify the relevant groups which are not covered by our theoretical results, and thus need to be considered in our computer search.
We realized this stage using GAP \cite{GAP4}.

The second stage is a \emph{preprocessing stage}, also realized in GAP, during which we compute the information required for the search and store it in a JSON file.
Besides the multiplication table describing the group, the most important information computed in this stage comprises a set of \emph{forbidden elements} and a set of \emph{required subsets}.

Forbidden elements are elements which cannot be part of any geodetic generating set except for the complete one.
A required subset is a set of elements of which each geodetic generating set needs to contain at least one.
We will give more details on how we compute and use these below.

The third and final stage is the actual search, which we have implemented in Rust.
Obviously, enumerating all generating sets is infeasible even for relatively small groups.
For example, already for the symmetric group $S_5$, which has $120$ elements, there are $2^{72}$ potential generating sets (symmetric subsets not containing the identity element).
To circumvent this problem, we discard generating sets based on the theoretical results described in the previous sections.
We implement this using a binary counter for enumerating generating sets, which allows us to systematically skip the respective ranges.
For each of the remaining generating sets, we test whether or not the resulting Cayley graph is geodetic and report those that are.

\subsection{Filtering}
 
We used GAP \cite{GAP4} and its SmallGrp library \cite{SmallGrp} to obtain a list of all finite groups up to order $1024$ relevant to our search. (In a second run we repeated the experiment also filtering out all odd-order groups, which we report on below.)
When generating this list of groups, we ignored the following.

\begin{itemize}
    \item All abelian groups (\Cref{cor:abelian}).
    \item All groups with even-order center (\Cref{cor:evencenter}).
    \item All groups with a large center (\Cref{cor:bound_center} and \Cref{prop:bound_center4}).
    \item All groups with abelian index-2 subgroups (\Cref{thm:abelian_by_c2}).
    \item The groups with abelian index-3 subgroups covered by \Cref{thm:abelian_by_c3_centerless}.
    \item The nilpotent groups covered by \Cref{thm:nilpotent} and \Cref{cor:SpecialNilpotent2}.
\end{itemize}

Note that, while there are approximately $50\cdot 10^9$ groups up to order $1024$, most of those are $2$-groups, which all have an even-order center.
Excluding the $2$-groups there are only $1206579$ groups of order at most $1024$, and after excluding the groups in the above list only $3197$ groups remain.
We provide more details on the number of groups falling into each of the above categories in \Cref{tab:filtering}.

\begin{remark}
    The smallest group with a center larger than the bound from \Cref{cor:bound_center} is $C_{13} \times A_4$.
    Its center has order $13$ -- just above the bound, which in this case is $12$.
    
    To find an example for \Cref{prop:bound_center4}, we have to look a bit further.
    The proposition excludes 3 as a prime divisor of the order of the group, and, 5 and 7 are excluded as prime divisors of the index of the center.
    The smallest group covered by \Cref{prop:bound_center4} that is not already covered by \Cref{cor:bound_center} is $C_7 \times (C_{13} \rtimes C_4)$.
    Its center has order 7, just larger than the bound of $6$.
\end{remark}

During filtering we only verify those bounds on the size of the center based on \Cref{cor:bound_center} and \Cref{prop:bound_center4} (with $m_0 = 3$).
More precise bounds on the size of non-complete geodetic generating sets, and thus also on the size of the center, are computed within the initialization step of our search algorithm.
Therein, we utilize most results of \Cref{sec:bounds} (with optimal parameters). Due to tighter bounds, we were able to exclude an additional $240$ groups from the search.

\begin{table}[t]
  \centering
  \begin{tabular}{lr}
    \toprule
    \textbf{Total Number of Groups} & $1{,}206{,}579$ \\\midrule
    Abelian Groups & $2{,}034$ \\
    Groups with Center of Even Order & $1{,}200{,}151$ \\
    Groups with Abelian Subgroup of Index $2$ & $989$ \\
    Nilpotent Groups as in \Cref{cor:SpecialNilpotent2} & $170$ \\
    Nilpotent Groups as in \Cref{thm:nilpotent} & $18$ \\
    Groups with Abelian Subgroup of Index $3$ as in  \Cref{thm:abelian_by_c3_centerless} & $86$ \\
    Groups with Large Center as in \Cref{cor:bound_center} and \Cref{prop:bound_center4} & $274$ \\\midrule
    \textbf{Remaining Groups} & $3{,}197$ \\
    \bottomrule
  \end{tabular}
  \smallskip
  \caption{Number of groups of order up to $1024$ and excluding $2$-groups which are caught by the different filtering steps. The filtering is performed in the same order as in the table, and each group is only counted towards the first category it matches.}
  \label{tab:filtering}
\end{table}

\begin{remark}\label{rem:even_order_filtering}
    For filtering the groups of even order up to 2014, we use the same method.
    However, there is one special case: the groups of order $1536 = 3\cdot 2^9$.
    There are more than $4 \cdot 10^8$ of them~-- too many for running through the entire filtering stage.
    However, most of them have a normal Sylow $3$-subgroup.
    Because such groups have an even-order center, they can be safely ignored by \Cref{cor:evencenter}.
    Therefore, we only run the filtering procedure for the remaining groups, indexed $408526598$--$408641062$ in the 
    SmallGrp 
    library.
\end{remark}

\subsection{Preprocessing}\label{sec:preprocessing}

The main objective of the preprocessing stage is to compute the forbidden elements and the required subsets of each group.
As mentioned above, an element is forbidden if, whenever it is part of a geodetic generating set, the associated Cayley graph is necessarily complete.
Note that, since we filter out abelian groups in the filtering stage, here we do not need to consider the possibility that the Cayley graph is an odd cycle.
The set of forbidden elements comprises all elements $g \in G$ such that $h=g$ or $h=g^2$ is nontrivial and satisfies $h^G \sse \{h^{\pm1}\}$; see \Cref{lem:conjugacyclass} and \Cref{lem:commgen_square_pt2}.
In particular, this includes central elements (\Cref{cor:center}) and their square roots (\Cref{thm:center_geodesics}).

A required subset is a set of which each geodetic generating set needs to contain at least one element.
In the preprocessing we compute the following sets, which we know to be required.
\begin{itemize}
    \item Each conjugacy class of elements of order two (\Cref{lem:conjugate_ord_two}).
    \item Each normal subgroup of index two (\Cref{lem:index_two}).
    \item Each complement of a maximal subgroup (as we want a generating set).
\end{itemize}
There is one other family of required subsets: the potential roots of each central element.
However, since these sets are smaller, and thus their inclusion is more effective, when we know the length of central geodesics, we do not add these sets in the preprocessing stage, but later in the search algorithm.

At this stage we also take advantage of automorphisms to reduce the size of the required subsets and thus reduce the number of generating sets we need to look at.
The procedure is to go through the required subsets one by one.
For each, we compute the orbits of its elements under the automorphism group and select one element from each orbit.
As we are interested only in symmetric generating sets, we consider the orbits of an element and its inverse as a single orbit.
Continuing with the next required subset, we no longer use the full automorphism group, but only the point-wise stabilizer of the elements that were selected in the previous required subsets.
This is justified by the following straight-forward observation:

\begin{lemma}
    Let $R \sse G$ and $Y \sse G$ and let $\tilde R \sse R$ be a system of representatives of $R/\operatorname{Stab}(Y)$.\footnote{Here, by abuse of notation, we write $R/\operatorname{Stab}(Y)$ to denote the set $R/{\sim}$ where $x \sim y$ if $\phi(x) = y$ for some $\phi \in \operatorname{Stab}(Y)$.}
    Then for all $X \sse G$ with $R \cap X \neq \emptyset$, there exists some $\phi \in \operatorname{Stab}(Y)$ such that $\phi(X) \cap \tilde R \neq \emptyset$.
\end{lemma}

In cases where the required subset contains only part of an orbit under the action of the automorphism group, we take care to select representatives that are part of the original set.
This way, for each generating set that contains an element of each of the original required subsets, there is a generating set which contains at least one element of each of the smaller required subsets obtained after applying the automorphism group.

Finally, we discard required subsets that are supersets of smaller ones.

\begin{remark}\label{rem:split}
    For the finite simple groups $\mathrm{PSL}(2,q)$ with $q \in \{17,19,16\}$, for parallelization, we split the computation of the search algorithm into several chunks. 
    This is implemented by generating several instances during the preprocessing stage with different required and forbidden subsets.
\end{remark}

\subsection{The search algorithm}

For a group $G$, we fix a subset $C\sse G\setminus \{1\}$ such that $\abs{\{g, g^{-1}\} \cap C} = 1$ for each $g\in G\setminus\{1\}$.
We call the elements of $C$ \emph{candidates}. 
In this way, each inverse-closed subset $\Sigma \sse G \setminus \{1\}$ (\ie each potential generating set) bijectively corresponds to a \emph{candidate set} $X = \Sigma \cap C \subseteq C$.
Note that if more than half of the elements of $G$ are of order two, then, according to Liebeck and MacHale \cite{Liebeck1972}, $G$ has an abelian subgroup of index two or $Z(G)$ has even order.
Since such groups were excluded during the filtering stage, we may assume at most half the elements of $G$ have order two; hence, $\abs{C} \leq \frac{1}{2}\abs{G}+\frac{1}{4}\abs{G}=\frac{3}{4}\abs{G}$. 

In the following, we identify $C$ with the set of numbers $\{0, \dots, \abs{C}-1\}$; in particular, we fix an order on~$C$.
We enumerate the potential generating sets of a group by enumerating all subsets of $C$ in the order of a binary counter from $0$ to $2^{\abs{C}}$, where a binary number naturally corresponds to a subset of $C$.
The main loop of our search algorithm is presented in \Cref{alg:main_loop}.
Incrementing the binary counter is done in lines 6 to 14. In line 18 we check whether the resulting Cayley graph is geodetic and connected (\testgeo{}, for details see below).
We use a stack to keep track of the current candidate
set $X$, the corresponding increment $I$, and some additional information that is not displayed in \Cref{alg:main_loop}.
During the execution of the algorithm, we maintain the following invariants of the stack (which we consider to grow upwards).
\begin{enumerate}[{label=(I${}_\arabic*$)}]
    \item \label{invariantI} If $(\canset,I)$ is above $(\canset',I')$, then $\canset \supseteq \canset'$ and $I < I'$.
    Moreover, $\canset \cap [I', \abs{C}] = \canset' \cap [I', \abs{C}]$.
    \item If $(\canset,I)$ is directly above $(\canset',I')$, then additionally $\canset \cap [I, \abs{C}] = (\canset' \cap [I, \abs{C}]) \cup \{I\}$.
\end{enumerate}

We use several pruning methods to shortcut the counting process. 
The variable $I_{\mathrm{next}}$ in \Cref{alg:main_loop} serves as the index of the bit to be increased next (meaning that $I_{\mathrm{next}}=0$ yields a usual increment by one); by increasing $I_{\mathrm{next}}$, we can skip over certain values for the counter.

The pruning methods, described in detail below, 
rely upon 
\begin{itemize}
    \item bounds on the size of the generating set, 

    \item the forbidden candidate array (for simplicity not included in the pseudocode \Cref{alg:main_loop}),

    \item the required subsets described above in \cref{sec:preprocessing} (handled in line 5 in \Cref{alg:main_loop}), 

    \item saturating the generating set (\fillin),
    \item handling of collisions at distance 3 discovered during the check whether the Cayley graph is geodetic (\handlethreecolls).
\end{itemize}

Finally, for groups of even order and groups with non-trivial center, we employ further pruning techniques during \testgeo and \handlethreecolls{}.

\begin{algorithm}
\begin{minipage}[t]{.45\textwidth}
\small\textbf{procedure} \textsc{CheckGroup}($G$)
\begin{algorithmic}[1]\small
  \State $(\canset, I) \gets (\emptyset, \infty)$
  \State Stack.\textsc{Push}($\canset, I$)
  \State $I_{\mathrm{next}} \gets 0$

  \While{Stack $\neq \emptyset$}
    \State $I_{\mathrm{next}} \gets \max \{I_{\mathrm{next}}, \textsc{Required}(\canset)\}$
    
    \While{$I_{\mathrm{next}} \in \canset$}
        \State $I_{\mathrm{next}} \gets I_{\mathrm{next}} + 1$
    \EndWhile
    \If{$I_{\mathrm{next}} \geq \abs{C}$}
        \State\Return
    \EndIf
    \While{$I_{\mathrm{next}} \geq I$}
        \State $(\canset,I) \gets {}$Stack.\textsc{pop}()
    \EndWhile
    \State $I_{\mathrm{last}} \gets I$ 
    
    \State $(\canset, I) \gets (\canset \cup \{I_{\mathrm{next}}\}, I_{\mathrm{next}})$
    \State Stack.\textsc{Push}$(\canset ,\text{$I$})$
    \State $I_{\mathrm{next}} \gets \textsc{CheckGenSet}(\canset, I, I_{\mathrm{last}})$
  \EndWhile
  \algstore{searchalg}
\end{algorithmic}
\end{minipage}
\begin{minipage}[t]{.5\textwidth}
\small\textbf{procedure} \textsc{CheckGenSet}($\canset, I, I_{\mathrm{last}}$)
\begin{algorithmic}[1]\small\algrestore{searchalg}
    \If{\fillin{}($\canset, I, I_{\mathrm{last}}$) fails}
       \State\Return $I$
    \EndIf

    \State $T \gets {}$\testgeo{}($\canset$)
    \If{$T = \textsc{Geodetic}$}
        \State\textbf{output} $\canset$
        \State\Return $0$
    \EndIf
    \If{\handlethreecolls{}($T, \canset, I, I_{\mathrm{last}}$) fails}
       \State\Return $I$
    \Else
        \State\Return $0$
    \EndIf
\end{algorithmic}
\end{minipage}
\caption{Outline of our search algorithm.}
\label{alg:main_loop}
\end{algorithm}

\subsubsection*{Bounds on the number of generators.}
In \Cref{sec:bounds} we developed several bounds on the size of non-complete geodetic generating sets.
As \cref{prop:notCG} and \cref{lem:small_order_diameter_two_implies_c5} already cover the diameter-two case of our search, the most general such bound is due to \Cref{thm:bound_generating_set}: 
it gives us an upper bound of $c \cdot \sqrt{n}$ where $n$ is the order of the group.
The constant factor $c$ is at most $2 / \sqrt{5}$ but, depending on the size of the group, the factor can be even smaller.
In fact, we compute the optimal bound given by \Cref{thm:bound_generating_set} in our program.
If the group has non-trivial center, we use \Cref{thm:bound_generating_set+center} to obtain an even smaller bound in the order of $\sqrt[3]{n}$.
We use these bounds to prune the search tree as follows:
whenever there are too many bits set to one in the counter (\ie $\abs{X}$ is too big~-- detected either in \fillin or \handlethreecolls), we increment the least-significant bit currently set to one by setting $I_{\mathrm{next}} \gets I$ (see lines 17 and 23 of \Cref{alg:main_loop}).

In some cases, the bound computed at this stage is impossible to satisfy and, therefore, the group possesses no geodetic generating set other than the complete one.
In this way, we exclude an additional $240$ (of the remaining $3197$) groups from the search.

\subsubsection*{Forbidden candidates.}
As a first improvement, we use a bit array indicating candidates that are forbidden given the other candidates already contained in $\canset$.
For each entry on the stack (stack frame) we keep a separate forbidden candidate array. 
Initially, the forbidden candidate array of the first stack frame comprises the forbidden candidates computed during preprocessing (\cref{sec:preprocessing}).
When creating a new stack frame, the forbidden candidates of its predecessor are copied; afterwards, additional candidates might be marked as forbidden on the new stack frame.
An additional candidate can be marked as forbidden if its inclusion in $\canset$ would 
\begin{itemize}
    \item lead to a generating set that is too large,
    \item require the inclusion of a candidate that is already forbidden, or
    \item violate the order in which we search through the generating sets.
\end{itemize}
These conditions are tested for during the \fillin and \handlethreecolls procedures.
For details we refer to the respective paragraphs below.

We use the forbidden candidate array in the \fillin and \handlethreecolls procedures as well as in the counter logic.
While omitted from the description in \Cref{alg:main_loop} for simplicity, its incorporation is rather straightforward: whenever $I_{\mathrm{next}}$ is a forbidden candidate, it is incremented.

\subsubsection*{Required subsets.}

We incorporate the required subsets computed during the preprocessing step as detailed in \Cref{sec:preprocessing}.
Whenever there is a required subset $R$ with $R \cap X = \emptyset$, we increase $I_{\mathrm{next}}$ to the smallest candidate contained in $R$ (see line 5 of \Cref{alg:main_loop}).
This dispenses with all candidate sets in between which do not contain an element of the required subset and, thus, need not be considered.

If multiple required subsets are disjoint from $X$, we apply the following heuristic: 
choose a required subset to be satisfied first such that the smallest candidate contained in it is maximal with respect to the order~on~$C$.

It may happen that the smallest candidate of a required subset is forbidden. 
In this case we simply consider the smallest candidate $I$ in the required subset that is not forbidden. 
If $I$ is larger than $I_\mathrm{last}$, the candidate of the previous stack frame, then we cannot satisfy the required subset by pushing a new stack frame without violating the order in which we search through the generating sets (Invariant~\ref{invariantI}).
Therefore, we assign $I_{\mathrm{next}} \gets I_{\mathrm{last}}$ in order to skip a number of candidate sets to which we cannot add any candidate from the required subset.

Finally, we remark that further required subsets are created during the \handlethreecolls procedure and at the beginning of the search for groups with non-trivial center.

\subsubsection*{Saturating the generating set \textnormal{(\fillin)}.} 

The most crucial improvement of our algorithm over a na\"ive search is based on \Cref{lem:diamonds}.
If $a, b, c, d \in \Sigma$ with $ab = cd \neq 1$ for a geodetic generating set $\Sigma \sse G$, then $(ab), (c^{-1}a) \in \Sigma$.
Therefore, whenever we add a new candidate to $\canset$, we test whether a product of the newly added element(s) and some other element of the generating set $\Sigma$ associated with $X$ equals a different product of two elements of $\Sigma$.
If so, we also add this product to $\Sigma$ by adding the corresponding candidate to $X$.
By repeating this we obtain a candidate set $X'$.
If the corresponding generating set $\Sigma'$ is too large, or if $X'$ contains a forbidden candidate, then the \fillin procedure fails and marks the current increment $I$ as forbidden on the previous stack frame.

In order to avoid considering the same generating set multiple times, we only want to add a candidate $J$ to $\canset$ during the \fillin procedure if $J < I$; otherwise, the resulting candidate set would be considered again after incrementing the counter.

As such, if $X'$ contains a candidate $J$ with $I < J$, then the \fillin procedure fails and, thus, in line 17 of \Cref{alg:main_loop} we set $I_{\mathrm{next}} \gets I$ skipping a number of candidate sets that would violate the search order.
Moreover, if $X'$ contains a candidate $J$ with $I_{\mathrm{last}} < J$, then we also mark $I$ as forbidden on the previous stack frame, which corresponds to $I_{\mathrm{last}}$.
This is because introducing $I$ as a candidate in any stack frame above the previous one would violate the second part of Invariant~\ref{invariantI}.

\subsubsection*{Testing whether the Cayley graph is geodetic \textnormal{(\testgeo)}.}
As the Cayley graph is vertex-transitive, it suffices to check whether geodesics from the origin $1 \in G$ to each other vertex exist and are unique.
We implemented this using a breadth-first search.
If during this breadth-first search we encounter an element with two different geodesics of length three, then we record this element for later handling.
We collect a certain number of such elements depending on the order of the group and other parameters such as, e.g., the order of the center and the number of elements of order two.
In the case of groups with non-trivial center we implemented some further checks as detailed below.

\subsubsection*{Handling collisions at distance three \textnormal{(\handlethreecolls)}.}
While testing whether the Cayley graph is geodetic, we compute a list of elements which have multiple geodesics of length three.
We aim to extend the generating set to a geodetic generating set.
To achieve this, we need to add new generators (by adding the respective candidates) such that each element in the list either becomes a generator or can be uniquely written as a product of two generators.

Since there are many options as for which generators to add, for each of the recorded collisions, we create a temporary required subset. 
It remains valid until the current candidate is removed from the candidate set, \ie the lifetime of these new required subsets is tied to the current stack frame.

The temporary required subset associated with a collision is constructed as follows.
Suppose that adding a single candidate $J$ to $X$ would resolve the collision.
Then we insert $J$ into the required subset, unless invoking \fillin with $X' = X \cup \{ J \}$ fails.
For pairs of candidates whose addition would resolve the collision, we proceed similarly, but only add the larger of the two candidates to the required subset.

Finally, if all calls to \fillin{} fail, and thus the required subset is empty, then the call to \handlethreecolls fails as we cannot resolve the collision.
Crucially, if all calls to \fillin{} forbid the corresponding inclusion, then we also forbid $I$ on the stack frame corresponding to $I_\mathrm{last}$.

\subsubsection*{Modifications for groups of even order.}

If during \testgeo we find an element of order two among the collisions at distance three, then the only choice is to add the element itself to the generating set.
Recall that, by \Cref{lem:conjugate_ord_two}, such an element cannot have a geodesic of length two.
As such, we do not construct a required subset in this case (by calling \handlethreecolls), but add the respective candidates directly to the candidate set instead (or possibly forbid the current increment $I$) and continue afterwards with \fillin.

Based on the same observation, as a further improvement for groups of even order (more precisely, for groups whose order is divisible by six) we have implemented the following.
During \fillin, we check whether a generator is of order six, and thus generates a subgroup isomorphic to $C_6$, or whether two generators together generate a subgroup isomorphic to $S_3$.
If so, we add all non-trivial elements of the corresponding subgroup to the generating set $\Sigma$ since such a subgroup is complete with respect to any geodetic generating set as the following lemma shows.

\begin{lemma}\label{lem:saturate_c6_d6}
Let $\mathrm{Cay}(G, \Sigma)$ be geodetic and $H \leq G$ with $\gen{H \cap \Sigma} = H$. 
If $\abs{H} = 6$, then $H$ is complete.
\end{lemma}
\begin{proof}
Suppose that $H \cong C_6$.
If we have a generator $g \in H \cap \Sigma$ of order six, then $g^3 = {\ov g}^{3}$ has length one by \cref{lem:conjugate_ord_two}, \ie $g^3 \in \Sigma$.
It follows that $H \cap \Sigma$ always contains some element $h$ of order two and at least one other element $k$ which commutes with $h$.
Hence, the subgroup $H = \gen{h, k}$ is complete by \cref{cor:commgen_subgroup}. 

Similarly, let $g,h \in \Sigma$ generate the subgroup $H$ isomorphic to $S_3$ and assume (without loss of generality) that $g$ is of order two. 
If $h$ is of order three, then we have $gh = h^{-1}g$ and, hence, $gh \in \Sigma$ and similarly $hg \in \Sigma$ and $(hg)h = h(gh) \in \Sigma$ showing that $H = \{1, g, h, gh, hg, hgh\}$ is complete.
If $h$ is of order two, then $ghg = hgh$ is of order two; hence, by \cref{lem:conjugate_ord_two}, $ghg \in \Sigma$. Thus, we have an entire conjugacy class $g^H = \{g,h, ghg\}$ contained in $\Sigma$.
Therefore, the subgroup $H = \gen{g,h}$ is complete by \cref{lem:conjugacyclass}.
\end{proof}

\subsubsection*{Modifications for groups with non-trivial center.}
For groups with non-trivial center, we compute all possible lengths of central geodesics given by~\Cref{thm:center_geodesics}.
Then we run the search repeatedly, once for each length of central geodesics.
We add the corresponding central roots as required subsets and add the central roots which are too short as forbidden elements.
When testing whether the Cayley graph is geodetic, we additionally check whether there are central elements with geodesics shorter than the length of central geodesics.
If that happens, then the only geodetic graphs the Cayley graph can be extended to by adding more generators are graphs with a smaller length of central geodesics.

\subsection{Experimental results}\label{sec:exp_results}
Our experiments were conducted on a machine with an AMD Ryzen~9 5900X CPU (12 cores, 24 threads, 3.7\,GHz) and 128\,GB of RAM.
Running the experiments for all groups up to order 1024 took 299 hours of total CPU time, and we were able to establish

\ThmB*

With parallelization, the running time was dominated by a few ``difficult'' groups.
See \Cref{tab:experiments} for the running time of select groups and \Cref{fig:running_time} for a plot of the running time of every group of order up to 1024 compared to the group's order.

\begin{figure}[ht]
    \includegraphics{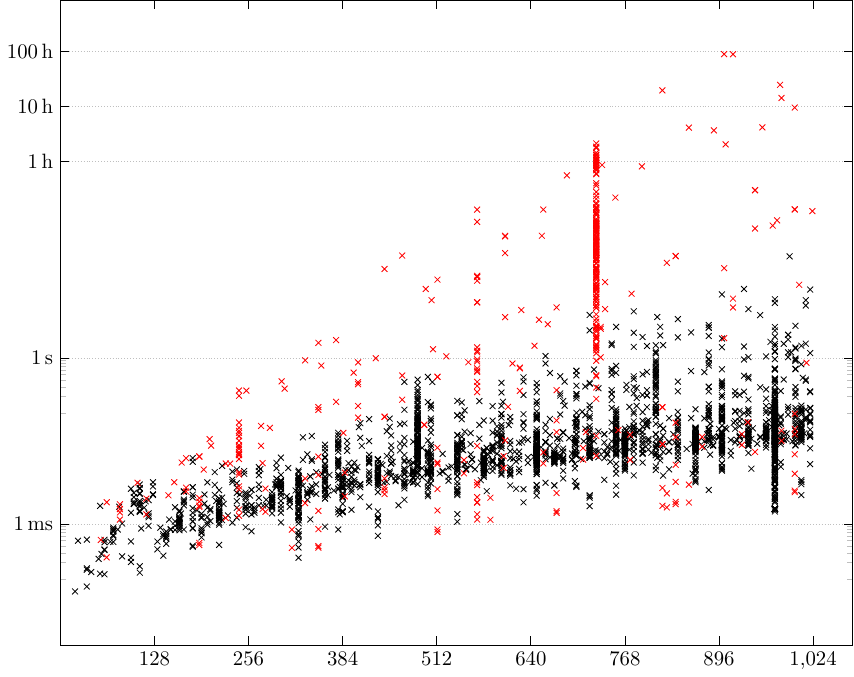}
    \caption{Running time of our computer search for all groups of order up to $1024$;
    groups of even order are marked in black, those of odd order are marked in red.}\label{fig:running_time}
\end{figure}

As it can be seen from \Cref{tab:experiments} and \Cref{fig:running_time}, our search is substantially faster for groups of even order.
Therefore, for groups of even order we extended our search up to order 2014.
We did not go beyond that since larger groups are not completely listed in the GAP 
SmallGrp
library \cite{SmallGrp}. 
Moreover, note that, as detailed in \cref{rem:even_order_filtering}, for groups of order 1536, we had to take some special care during the filtering stage.
\begin{theorem}\label{thmEVENORDER2014}
    \Cref{conj:main} holds for all groups of even order at most $2014$.
\end{theorem}

In the light of \Cref{cor:commgen_subgroup} and \Cref{thm:nilpotent} it seems reasonable to search for counterexamples to \Cref{conj:main} within classes of groups that are far from commutative.
Thus, it is natural to consider non-abelian finite simple groups.
Nevertheless, in our experiments these groups turned out to be even easier to handle than many of the other groups of even order. (Note that by the famous Feit-Thompson theorem \cite{FeitT} all non-abelian simple groups have even order.)
Indeed, for this special case we could go further than order $2014$ and succeeded to show the following.

\begin{theorem}\label{thm:AdditonalLargeGroups}
    \Cref{conj:main} holds for $S_7$ and for the simple groups $\mathrm{PSL}(2,17)$, $A_7$, $\mathrm{PSL}(2,19)$, and $\mathrm{PSL}(2,16)$.
    In particular, it holds for all simple groups of order at most $5000$.
\end{theorem}

To give an indication which groups are ``easy'' and which are ``difficult'', we give an overview on some selected groups in \Cref{tab:experiments}. The table contains the following classes of groups:
    \begin{itemize}
      \item alternating and symmetric groups,
      \item further simple groups of order exceeding 2014, 
      \item the five groups of order up to 1024 
      with the longest running time,
      \item different groups of order $729 = 3^6$, 
      \item groups with the longest running time among the even groups of order up to $1024$,
      \item groups of even order between 1026 and 2014 with the longest running time.
  \end{itemize}

\begin{table}[h!]
  \centering
  \begin{tabular}{rrrrl}
    \toprule
    \textbf{Group} & \textbf{Order} & \textbf{Index} &\qquad 
    \textbf{Sets tested} & \textbf{Duration} \\
    \midrule
        $A_5$ & $60$ & $5$ & 
        $31$ & $0.3$\,ms \\
        $S_5$ & $120$ & $34$ & 
        $207$ & $1.3$\,ms \\
        $A_6$ & $360$ & $118$ & 
        $6249$ & $23$\,ms \\
        $S_6$ & $720$ & $763$ & 
        $27590$ & $120$\,ms \\
        $A_7$ & $2{,}520$ & --- & 
        $74{,}946{,}283$ & $52$\,min \\
        $S_7$ & $5{,}040$ & --- & 
        $1{,}059{,}510{,}737$ & $197$\,min \\
        \midrule
        $\mathrm{PSL}(2,17)\mathrlap{{}^\Vert}$ & $2{,}448$ & --- &
        $123{,}451{,}769$ & $23$\,min\\
        $\mathrm{PSL}(2,16)\mathrlap{{}^\Vert}$ & $4{,}080$ & --- &
        $4{,}869{,}673{,}337$ & $13$\,h\\
        $\mathrm{PSL}(2,19)\mathrlap{{}^\Vert}$ & $3{,}420$ & --- &
        $2{,}696{,}472{,}513$ & $21$\,h\\
        \midrule
        $C_{109} \rtimes C_9$ & $981$ & $3$ & 
        $5{,}683{,}264{,}056$ & $14$\,h \\
        $(C_7 \rtimes C_3) \times (C_{13} \rtimes C_3)$ & $819$ & $6$ & 
        $8{,}728{,}959{,}134$ & $19$\,h \\
        $C_{89} \rtimes C_{11}$ & $979$ & $1$ & 
        $10{,}178{,}934{,}027$ & $24$\,h \\
        $C_{61} \rtimes C_{15}$ & $915$ & $1$ & 
        $43{,}174{,}839{,}011$ & $87$\,h \\
        $C_{43} \rtimes C_{21}$ & $903$ & $1$ & 
        $43{,}967{,}855{,}355$ & $88$\,h \\
        \midrule
        $C_9^2$ extended by $C_3^2$ & $729$ & $96$ &
        $13{,}525$ & $39$\,ms\\
        $(C_3 \times (C_{27} \rtimes C_3)) \rtimes C_3\mathrlap{{}^\ddagger}$  & $729$ & $90$ & 
        $11{,}276{,}468$ & $208$\,s \\
        $C_{27} \rtimes C_{27} \mathrlap{{}^\ddagger}$ & $729$ & $22$ & 
        $10{,}776{,}997$ & $255$\,s \\
        $(C_9 \rtimes C_9) \rtimes C_9$ & $729$ & $75$ & 
        $97{,}944{,}803$ & $52$\,min\\
        $(C_{27} \rtimes C_9) \rtimes C_3$ & $729$ & $390$ & 
        $175{,}898{,}535$ & $87$\,min \\
        $((C_9 \times C_3) \rtimes C_3) \rtimes C_3^2$ & $729$ & $399$ & 
        $240{,}194{,}985$ & $128$\,min \\
        \midrule
        $C_{31} \rtimes C_{30}$ & $930$ & $1$ & 
        $821{,}601$ & $18$\,s \\
        $C_{17} \times A_5$ & $1{,}020$ & $9$ & 
        $1{,}030{,}890$ & $18$\,s \\
        $C_2^5 \rtimes C_{31}$ & $992$ & $194$ & 
        $3{,}260{,}710$ & $70$\,s \\
        \midrule
        $C_{11} \times \mathrm{PSL}(3,2)$ & $1{,}848$ & $127$ & 
        $1{,}024{,}861{,}064$ & $8$\,h \\
        $C_{37} \rtimes C_{54}$ & $1{,}998$ & $7$ & 
        $1{,}292{,}527{,}452$ & $13$\,h \\
        $C_{2}^4 \rtimes (C_5 \times (C_7 \rtimes C_3))$ & $1{,}680$ & $939$ & 
        $1{,}709{,}925{,}665$ & $14$\,h \\
    \bottomrule
  \end{tabular}
  \smallskip
  \caption{Experiments for selected groups including the five groups with the longest running time.
  The first column shows the group according to the GAP structure description, the second column displays the order of the group, the third column the index in the 
  SmallGrp
  library, 
  the fourth column the number of calls to the \testgeo procedure, and the final column contains the running time. 
  The computation was parallelized for the groups $\mathrm{PSL}(2,q)$ (marked with ${}^\Vert$) according to \cref{rem:split}; the quantities displayed in the last two columns are cumulative. 
  The center of the groups of order $729$ is either of order three or nine (marked with ${}^\ddagger$).
  }\label{tab:experiments}
\end{table}

Note that the order of the groups certainly plays a role in the running time, but there was also a huge variance in running time for different groups of roughly the same order.
For instance this can be seen very prominently for groups of order 729 in \cref{fig:running_time}.

There are various reasons for this difference. 
For example we observe an impact of the size of the center, which is to be expected given the results in \cref{sec:bounds}. 
Moreover, we see the clear effect of the groups having even order: for these we have additional possibilities to generate small required subsets and we can use the improved \fillin procedure.
However, also among groups of odd order and with trivial center there is a huge variation in running time. We observe that the most difficult instances are groups $C_{p} \rtimes C_{n}$ with a faithful action, $p$ prime, and $n$ as large as possible.

Additional statistics on our experiments may be found at

\centerline{\url{https://osf.io/9ay6s/?view_only=37e18301e4e74e12bfe4e07b90b924c0}.}

\section{Discussion}

We have shown that for several infinite classes of finite groups there are no geodetic Cayley graphs except the complete graphs. This includes all abelian groups (except cyclic groups of odd order), dihedral groups, and groups with even-order center, as well as many nilpotent groups. Moreover, we have verified by a computer search that \Cref{conj:main} holds for all groups up to order 1024, all groups of even order up to 2014, all simple groups of order up to $5000$, and the symmetric group~$S_7$.

The main open problem, of course, remains whether \Cref{conj:main} holds for all finite groups, \ie that every geodetic Cayley graph of a finite group is either complete or a cycle of odd length. 
Our experiments suggest that it might be reasonable to aim for proving \Cref{conj:main} for all groups of even order.

\section{Acknowledgments}
The authors sincerely thank Alexey Talambutsa for pointing out previous work, 
Filippo Prandina for highlighting a mistake in \cref{thm:bound_generating_set+simplified} in 
a preliminary version,
and the anonymous reviewers for their careful reading and helpful corrections, comments, and improvements to the writing of this article.
In completing this work, the first two authors were supported by Australian Research Council grant DP210100271 and the fifth author was partially supported by DFG grant WE 6835/1--2.

\bibliographystyle{plainurl}
\bibliography{main}

\end{document}